\pdfoutput=1 
\newif\ifpreprint
\preprinttrue 
\ifdefined\ispreprint
  \preprinttrue
\fi
\ifdefined\isjournal
  \preprintfalse
\fi

\documentclass[a4paper,10pt]{article}

\usepackage{authblk}

\usepackage[a4paper]{geometry}
\usepackage[utf8]{inputenc}

\usepackage{microtype}
\usepackage{lmodern}
\usepackage[T1]{fontenc}

\usepackage{amsmath,amsthm,amsfonts}
  \newtheorem{theorem}{Theorem}
  \newtheorem*{theorem1a}{Theorem 1a}
  \newtheorem*{theorem1b}{Theorem 1b}
  \newtheorem{lemma}{Lemma}
  \newtheorem{proposition}{Proposition}
  \newtheorem{definition}{Definition}
  \theoremstyle{remark}

\usepackage[hidelinks]{hyperref}

\usepackage{graphicx}

\usepackage{tikz,pgfplots}
\pgfplotsset{compat=1.3}
\usepgfplotslibrary{groupplots}
\usetikzlibrary{calc,trees,positioning,arrows,chains,shapes.geometric,%
    decorations.pathreplacing,decorations.pathmorphing,shapes,%
    matrix,shapes.symbols}
\tikzset{
  >=stealth',
  punktchain/.style={
    draw=black, very thick,
    text width=10em, 
    minimum height=3em, 
    text centered, 
    on chain},
  tuborg/.style={decorate},
  tubnode/.style={midway, right=2pt},
}

\usepackage{algorithm,algpseudocode}

\usepgfplotslibrary{external} 
 \usepackage[normalem]{ulem} 
 
\newcommand{\ri}{\textit{(i)}\ }
\newcommand{\rii}{\textit{(ii)}\ }
\newcommand{\riip}{\textit{(ii)}}

\newcommand{\supp}{\mathrm{supp}}
\newcommand{\imagunit}{\mathrm{i}}
\newcommand{\bszero}{\boldsymbol{0}} 
\newcommand{\bsone}{\boldsymbol{1}}  

\newcommand{\bsa}{{\boldsymbol{a}}}
\newcommand{\bsb}{{\boldsymbol{b}}}
\newcommand{\bsh}{{\boldsymbol{h}}}
\newcommand{\bsp}{{\boldsymbol{p}}}
\newcommand{\bsx}{{\boldsymbol{x}}}
\newcommand{\bsy}{{\boldsymbol{y}}}
\newcommand{\bsz}{{\boldsymbol{z}}}

\newcommand{\bsalpha}{{\boldsymbol{\alpha}}}
\newcommand{\bstau}{{\boldsymbol{\tau}}}

\newcommand{\setu}{{\mathfrak{u}}}
\newcommand{\setv}{{\mathfrak{v}}}
\newcommand{\setw}{{\mathfrak{w}}}

\newcommand{\rd}{\,\mathrm{d}}
\newcommand{\rmd}{\mathrm{d}}
\newcommand{\rmi}{\imagunit}

\newcommand{\Z}{\mathbb{Z}}
\newcommand{\R}{\mathbb{R}}
\newcommand{\N}{\mathbb{N}}

\newcommand{\calH}{\mathcal{H}}

\newcommand{\KKor}{K^{\mathrm{Kor}}}
\newcommand{\KKorab}{K^{\mathrm{Kor}, [\bsa,\bsb]}}
\newcommand{\KSob}{K^{\mathrm{Sob}}}
\newcommand{\KSobstar}{K^{\mathrm{Sob},*}}
\newcommand{\KSobab}{K^{\mathrm{Sob}, [\bsa,\bsb]}}
\newcommand{\KSobabone}{K^{\mathrm{Sob}, [a,b]}}

\newcommand{\tilB}{\widetilde{B}}

\definecolor{darkred}{RGB}{139,0,0}
\definecolor{darkgreen}{RGB}{0,130,70}
\definecolor{darkorange}{RGB}{255,140,0}
\definecolor{darkorange}{RGB}{180,60,0}

\allowdisplaybreaks

\title{Scaled lattice rules for integration on $\R^d$ achieving \\ higher-order convergence with error analysis \\ in terms of orthogonal projections onto periodic spaces}

\author[1]{Dirk Nuyens}
\author[2]{Yuya Suzuki}

\affil[1]{dirk.nuyens@cs.kuleuven.be, KU~Leuven, Belgium}
\affil[2]{yuya.suzuki@ntnu.no, NTNU, Norway}

\begin{document}
 
\maketitle

\begin{abstract}
We introduce a new method to approximate integrals $\int_{\R^d} f(\bsx) \rd\bsx$ which simply scales lattice rules from the unit cube $[0,1]^d$ to properly sized boxes on $\R^d$, hereby achieving higher-order convergence that matches the smoothness of the integrand function $f$ in a certain Sobolev space of dominating mixed smoothness.
Our method only assumes that we can evaluate the integrand function $f$ and does not assume a particular density nor the ability to sample from it.
In particular, for the theoretical analysis we show a new result that the method of adding Bernoulli polynomials to a function to make it ``periodic'' on a box without changing its integral value over the box, is equivalent to an orthogonal projection from a well chosen Sobolev space of dominating mixed smoothness to an associated periodic Sobolev space of the same dominating mixed smoothness, which we call a Korobov space.
We note that the Bernoulli polynomial method is often not used because of its excessive computational complexity and also here we only make use of it in our theoretical analysis.
We show that our new method of applying scaled lattice rules to increasing boxes can be interpreted as orthogonal projections with decreasing projection error.
Such a method would not work on the unit cube since then the committed error caused by non-periodicity of the integrand would be constant, but for integration on the Euclidean space we can use the certain decay towards zero when the boxes grow.
Hence we can bound the truncation error as well as the projection error and show higher-order convergence in applying scaled lattice rules for integration on Euclidean space.
We illustrate our theoretical analysis by numerical experiments which confirm our findings.
\end{abstract}

\section{Introduction}\label{sec:intro}

Lattice rules are one important branch of quasi-Monte Carlo (QMC) methods traditionally applied in the approximation of multivariate integrals over the unit cube $[0,1]^d$ with periodic integrand functions, see, e.g., the classical references \cite{N1992,SJ1994}.
In many applications integration over $\R^d$ arises.
In this paper we study how to apply lattice rules in this setting with the aim of obtaining higher-order convergence, i.e., to obtain error bounds which show that the error can be bounded by $O(n^{-\alpha})$ for $\alpha > 1$, with $n$ the number of integrand evaluations.
Typically the smoothness of the integrand is related to $\alpha$ by the number of derivatives that exist (in all directions) and are, e.g., $L_2$ integrable (see further for details).
More specifically, we are interested in using lattice rules to approximate the integral by first truncating to a box $[\bsa,\bsb] := \prod_{j=1}^d [a_j, b_j]$, with all $a_j \le b_j$, and then mapping the lattice rule to the box:
\[
 \int_{\R^d} f(\bsx) \rd \bsx
 \quad\approx\quad
 \int_{[\bsa,\bsb]} f(\bsx) \rd \bsx
 \quad\approx\quad
 \frac{\prod_{j=1}^d (b_j-a_j)}{n} \sum_{i=1}^n f(\bsp_i^{[\bsa,\bsb]})
 .
\]
The quadrature nodes $\bsp_i^{[\bsa,\bsb]}$ are obtained from a rank-$1$ lattice point set $\Lambda(\bsz,n)$, defined by two parameters, the number of points $n$ and the generating vector $\bsz \in \Z_n^d$,
\[
 \Lambda(\bsz,n)
 :=
 \left\{ \bsp_i:= \frac{i\bsz}{n} \bmod{1} \right\}_{i=1}^n
 \subset
 [0,1]^d
 ,
\]
and the nodes are mapped from the unit cube $[0,1]^d$ to the box $[\bsa,\bsb]$ by
\[
 p_{i,j}^{[\bsa,\bsb]}
 :=
 (b_j - a_j) \, p_{i,j} + a_j
\]
where $p_{i,j}^{[\bsa,\bsb]}$ is the $j$-th component of $\bsp_i^{[\bsa,\bsb]}$, and likewise, $p_{i,j}$ is the $j$-th component of $\bsp_i$.
By the triangle inequality we have
\begin{multline}\label{eq:split-error-numerics}
  \left| \int_{\R^d} f(\bsx) \rd \bsx  - \frac{\prod_{j=1}^d (b_j-a_j)}{n}\sum_{i=1}^n f(\bsp_i^{[\bsa,\bsb]}) \right|
  \\\le
  \left| \int_{\R^d} f(\bsx) \rd \bsx  - \int_{[\bsa,\bsb]} f(\bsx) \rd \bsx \right|
  +
  \left| \int_{[\bsa,\bsb]} f(\bsx) \rd \bsx - \frac{\prod_{j=1}^d (b_j-a_j)}{n}\sum_{i=1}^n f(\bsp_i^{[\bsa,\bsb]}) \right|
  ,
\end{multline}
i.e., the total error is bounded by the sum of the truncation error and the integration error.
We will select the box $[\bsa,\bsb]$ and the number of lattice points $n$ to balance both errors.
Our major concern however will be how to analyse the error of integrating the non-periodic integrand on the box $[\bsa,\bsb]$ by a lattice rule.

Different approaches exist in the literature to apply quasi-Monte Carlo methods designed for integration over the unit cube $[0,1]^d$ to integration over $\R^d$, e.g., \cite{DILP2018,NK2014,NN2017,KSWW2010,IL2015,NN2020,KN2016}.
In \cite{NN2017} analytic functions with exponential decay in the physical and in the Fourier space are studied and exponential convergence is obtained using a regular grid with appropriate scaling in each direction.
In \cite{DILP2018} integration with respect to the Gaussian measure is considered where integrand functions are in some unanchored Sobolev spaces. Therein, higher-order convergence is achieved using linearly transformed higher-order digital nets on a box. Similarly, in \cite{NN2020} higher-order convergence is achieved with interlaced polynomial lattice rules where a different type of function spaces, namely anchored Sobolev spaces, are considered. In \cite{NK2014} randomly shifted lattice rules are studied in combination with an inverse mapping of the considered cumulative density function, and first order convergence is achieved.

In the present paper, we consider linearly transformed lattice rules on a box of increasing size, without random shifting or inverse mapping, and we derive explicit conditions to obtain higher-order convergence.

Apart from the problem of adjusting from the unit cube $[0,1]^d$ to $\R^d$, we also need to handle that the integrand function $f$ on the box $[\bsa,\bsb]$ is non-periodic.
The analysis of lattice rules for numerical integration is typically done in a periodic setting with absolutely converging Fourier series expansions.
In this setting the smoothness $\alpha$ as previously defined, means that the function values and its derivative values (up to a certain order) will match on the  boundaries of the unit cube.
To apply lattice rules in the non-periodic setting on the unit cube $[0,1]^d$ one can apply so-called ``periodizing transformations'' \cite{SJ1994,S1993,KSW2007}.
This method applies a variable substitution to the integral, dimension-by-dimension, to end up with a periodic integrand function while preserving the value of the integral. The problem with this strategy is that this ``blows up the norm'', as noted by many researchers \cite{H2002,GSY2019,Z1972,K1963}.
A second approach is known as the ``method of Bernoulli polynomials'' \cite{Z1972,K1963,SJ1994}. In this method we add Bernoulli polynomials to the original integrand in such a way that function values (and derivative values) match at the boundaries of the unit cube, making the integrand periodic, while preserving the value of the integral, since those Bernoulli polynomials integrate to zero.
As pointed out by many researchers \cite{Z1972,K1963,SJ1994,BH1992}, this method does not blow up the norm, but the amount of terms that need to be added increases exponentially (by correcting the value of the original integrand on all possible boundaries of the hypercube, and, by the need to analytically calculate the derivatives of the original integrand, which might also grow exponentially by, e.g., the need to use the chain rule). This is probably the reason why this last method appears to be not used  as much.
Both the periodizing transformation and the method of Bernoulli polynomials can transform a non-periodic function of smoothness $\alpha$ into a periodic function with the same smoothness, hence achieving $O(n^{-\alpha})$.
Since the periodizing transformation might blow up the norm and the method of Bernoulli polynomials could increase the cost of evaluating the integrand, the recent literature on lattice rules mostly applies two other techniques: ``random shifting'' and the ``tent-transform'', which can also be combined \cite{H2002,DNP2014,CKNS2016,GSY2019}.
This method is much easier and does not have the previously mentioned defects, but the convergence is limited to $O(n^{-1})$ and $O(n^{-2})$.

Here we take yet another approach which has not been studied yet.
First we analyse what the effect is of integrating a non-periodic function with a lattice rule in terms of a measure of how non-periodic the function is.
We can measure the non-periodicity of a function by quantifying how much the function and its derivatives (up to a certain order) differ at the boundaries of the domain.
We derive an explicit error bound taking this difference into account.
Since the method of Bernoulli polynomials actually corrects this non-periodicity, we use it in the theoretical analysis to get hold of the non-periodicity.
The advantage of this approach is that we never have to actually construct the modified integrand.
Of course, since, this defect is constant, such an approach would not work on the unit cube, or any fixed box.
However, in the setting of integration over $\R^d$ we can use the fact that as we increase the boundaries of our box $[\bsa,\bsb]$ to cover more and more of $\R^d$, then the function will differ less and less on the boundaries of our box.
If we demand a certain decay of the function values and its derivatives up to a certain order, then our integrand function becomes more and more periodic as our box increases.
It turns out that in our setting, where the integrand functions are from a certain reproducing kernel Hilbert space which we call an unanchored Sobolev space, the modified integrand which can be obtained by applying the Bernoulli polynomial method, is actually the orthogonal projection of the original function from the unanchored Sobolev space to an associated Korobov space, and this lets us apply the lattice rule without blowing up the norm, see Proposition~\ref{prop:projection} and Proposition~\ref{prop:projection-ab}.
Along the way we also show a new representation for functions in our Sobolev spaces, see Proposition~\ref{prop:representations} and Proposition~\ref{prop:representations-ab}.
Using the orthogonal projection we can thus decompose the integrand into a periodic part and a non-periodic part where those two parts are orthogonal in the unanchored Sobolev space.
Hence, for integration over $\R^d$, by using this orthogonal projection we can divide the analysis of the error into three parts: the truncation error, the integration error of the periodic part, and the projection error:
\begin{multline}\label{eq:split-error-analysis}
  \left| \int_{\R^d} f(\bsx) \rd \bsx  - \frac{\prod_{j=1}^d (b_j-a_j)}{n}\sum_{i=1}^n f(\bsp_i^{[\bsa,\bsb]}) \right|
  \\\hspace*{-5.4cm}\le 
  \left| \int_{\R^d} f(\bsx) \rd \bsx  - \int_{[\bsa,\bsb]} f(\bsx) \rd \bsx \right|
  \\+
  \left| \int_{[\bsa,\bsb]} F^{[\bsa,\bsb]}(\bsx) \rd \bsx - \frac{\prod_{j=1}^d (b_j-a_j)}{n}\sum_{i=1}^n F^{[\bsa,\bsb]}(\bsp_i^{[\bsa,\bsb]}) \right|
  \\
  +
  \left| \frac{\prod_{j=1}^d (b_j-a_j)}{n}\sum_{i=1}^n \left( F^{[\bsa,\bsb]}(\bsp_i^{[\bsa,\bsb]}) -  f(\bsp_i^{[\bsa,\bsb]}) \right) \right|
  ,
\end{multline}
where $F^{[\bsa,\bsb]}$ is the projected part of $f$ onto the Korobov space on the box $[\bsa,\bsb]$ and we have $\int_{[\bsa,\bsb]} f(\bsx) \rd \bsx = \int_{[\bsa,\bsb]} F^{[\bsa,\bsb]}(\bsx) \rd \bsx$.
To control the projection error, we will ask the integrand to become more and more periodic as the box increases.
The speed of becoming periodic will be carefully chosen such that the convergence of the total error is not deteriorated by the non-periodicity.
It is important to remark that the speed of becoming periodic is governed by the decay of the integrand function, and its partial derivatives, towards infinity, see Proposition~\ref{prop:projection-error}, and this same decay on the integrand function itself is also needed for the truncation error, see Proposition~\ref{prop:truncation}.

Combining Proposition~\ref{prop:projection-error} and Proposition~\ref{prop:truncation} with the integration error of the lattice rule, see Lemma~\ref{lem:cuberror-ab}, will lead us to our main result in Theorem~\ref{thm:total-error}.
As a simple example we state the following two results which follow from Theorem~\ref{thm:total-error}.
For the particular definition of the function spaces we refer to~\eqref{eq:Rd-norm} and Section~\ref{sec:RKHS}.
The lattice rule we refer to should be a ``good'' lattice rule for the standard Korobov space on the unit cube, see Lemma~\ref{lem:cuberror-ab}.

\begin{theorem1a}
Let $g$ have mixed smoothness $\alpha$ on any finite box $[\bsa,\bsb] \subset \R^d$ with all $|g^{(\bstau)}|$ for $\bstau \in \{0,\ldots,\alpha-1\}^d$ bounded by a polynomial and let $p_s$ be the product logistic density with scale parameter $s$, then the function $f(\bsx) = g(\bsx) \, p_s(\bsx)$ has mixed smoothness~$\alpha$ on $\R^d$ and satisfies our exponential decay condition~\eqref{eq:exp-decay} with $q=1$ and $\beta=1/s$.
  To approximate
  \begin{align*}
    \int_{\R^d} \underbrace{g(\bsx) \, p_s(\bsx)}_{= f(\bsx)} \rd{\bsx},
  \end{align*}
  our algorithm will select boxes
  \begin{align*}
    [-a,a]^d \quad \text{with} \quad a = a(n) = \alpha s \log(n)
  \end{align*}
  for a lattice rule with $n$ points.
  The error will converge like $O(n^{-\alpha + \epsilon})$ for arbitrary small $\epsilon > 0$.
\end{theorem1a}

\begin{theorem1b}
Let $g$ have mixed smoothness $\alpha$ on any finite box $[\bsa,\bsb] \subset \R^d$ with all $|g^{(\bstau)}|$ for $\bstau \in \{0,\ldots,\alpha-1\}^d$ bounded by a polynomial and let $\phi$ be the product normal density with variance $\sigma^2$, then the function $f(\bsx) = g(\bsx) \, \phi(\bsx)$ has mixed smoothness~$\alpha$ on $\R^d$ and satisfies our exponential decay condition~\eqref{eq:exp-decay} with $q=2$ and $\beta = 1/(2\sigma^2)$.
  To approximate
  \begin{align*}
    \int_{\R^d} \underbrace{g(\bsx) \, \phi(\bsx)}_{= f(\bsx)} \rd{\bsx},
  \end{align*}
  our algorithm will select boxes
  \begin{align*}
    [-a,a]^d \quad \text{with} \quad a = a(n) = \sqrt{2 \alpha \sigma^2 \log(n)}
  \end{align*}
  for a lattice rule with $n$ points.
  The error will converge like $O(n^{-\alpha + \epsilon})$ for arbitrary small $\epsilon > 0$.
\end{theorem1b}

The algorithm itself takes the super straightforward form presented in Algorithm~\ref{alg:scaled-lattice-rule}.
We present numerical examples for the two cases above in Section~\ref{sec:numerics}.
Similar results for a polynomial decay of the integrand function $f$ can also be found in Theorem~\ref{thm:total-error}.

Current research on QMC methods is often concerned with very high dimensional integrals and is interested in function space settings in which approximation of such integrals is tractable.
Tractability is mostly achieved by considering a sequence of positive weights which model the importance of different subsets of dimensions, see e.g., \cite{DSWW2004,SW1998,SW2001} and \cite[Section~4]{DKS2013}.
Our current method does not use such weighted spaces and therefore suffers from the curse of dimensionality.
We focus here on the problem of obtaining higher-order convergence for integration on $\R^d$ and leave the question of tractability for future work.

The rest of the present paper is organized as follows.
In Section~\ref{sec:RKHS} we introduce the needed reproducing kernel Hilbert spaces on the unit cube and also on a general box $[\bsa,\bsb] = \prod_{j=1}^d [a_j,b_j]$.
Section~\ref{sec:BPM-projection} introduces the orthogonal projection from our unanchored Sobolev space to the associated Korobov space.
Section~\ref{sec:error-analysis} gives the result for the total integration error using lattice rules for integration on $\R^d$. 
In Section~\ref{sec:numerics} we demonstrate our method with numerical examples.
Finally, Section~\ref{sec:conclusion} concludes this paper.

Throughout this paper,
$\R$ denotes the set of all real numbers, $\Z$ denotes the set of all integers, $\Z_n := \{0,1,\ldots,n-1 \}$ is the set of integers modulo $n$, $\N:=\{1,2, \ldots \}$ is the set of all natural numbers. By $\imagunit$ we mean the imaginary unit.
For $\setw \subseteq \{1,\ldots,d\}$ we write $[\bsa_\setw,\bsb_\setw] := \prod_{j\in \setw} [a_j,b_j]$ and $[\bsa_{-\setw},\bsb_{-\setw}]
= \prod_{j \in -\setw} [a_j,b_j]
:= \prod_{j \notin \setw} [a_j,b_j]
:= \prod_{j \in \{1,\ldots,d\} \setminus \setw}
[a_j,b_j]$ with $-\setw := \{1,\ldots,d\} \setminus \setw$. Similarly $\bsx_\setw := (x_j)_{j\in\setw}$ and $\bsx_{-\setw} := (x_j)_{j\notin\setw}$.

\section{Reproducing kernel Hilbert spaces}\label{sec:RKHS}

We consider integration for functions in unanchored Sobolev spaces.
In the context of high-dimensional integration and QMC methods these are basically tensor products of one-dimensional spaces where derivatives up to a certain order $\alpha$ are $L_2$ integrable; and where in the multi-variate setting these derivatives can be taken up to order $\alpha$ in each dimension simultaneously.
In different literature, such spaces are often said to have ``dominating mixed-smoothness''.
In our analysis, we connect this space to the Korobov spaces, which are similar function spaces of ``periodic functions'' and where it is known that lattice rules can achieve the optimal order of convergence \cite{N1992,SJ1994}.
We first introduce those spaces on the unit cube $[0,1]^d$ and then on a box $[\bsa,\bsb] \subset \R^d$.

In the classical references the function spaces are mostly Banach spaces.
Here we want to make use of the connection between the Hilbert spaces $L_2$ and $\ell_2$ with respect to Fourier coefficients by making use of the Parseval theorem to be able to express $L_2$ norms of derivatives as $\ell_2$ norms of the Fourier coefficients of the original function.
Our function spaces are therefore reproducing kernel Hilbert spaces (RKHS).
Secondarily, since we are now in a Hilbert space setting, we can later use the idea of an orthogonal projection and its complement to analyze the effect of integrating a non-periodic function by a lattice rule.

These function spaces are tensor-products of one dimensional spaces.
For the one-dimensional spaces we will make use of the fact that if $f$ is \emph{absolutely continuous} on $[a,b]$ then $f$ has a derivative $f'$ almost everywhere that is Lebesgue integrable on $[a,b]$, $-\infty < a < b < \infty$, and for which
\begin{align*}
  f(x)
  &=
  f(a) + \int_a^x f'(t) \rd{t}
  \qquad
  \text{for all } x \in [a,b]
  ,
\end{align*}
and in particular $\int_a^b f'(t) \rd{t} = f(b) - f(a)$, see~\cite{R1976}.
As a shorthand we sometimes write $\int_a^b f'(t) \rd{t} = \left. f(t) \right|_{t=a}^b = \left[ f(t) \right]_{t=a}^b = f(b) - f(a)$.
The one-dimensional Sobolev spaces we describe here will consist of $f$ for which $f^{(\tau)}$ will be absolutely continuous for all $\tau = 0, \ldots, \alpha-1$ and $f^{(\alpha)}$ will be $L_2([a,b])$ integrable.
What remains is to choose an inner product, which we will do in the next sections.

\subsection{Some properties of Bernoulli polynomials}\label{sec:bernoulli}

We recall some well known properties of Bernoulli polynomials, see, e.g., \cite{NIST:DLMF}, which will be denoted by $B_\tau$, where $\tau$ is the degree of the polynomial.
The first few Bernoulli polynomials are $B_0(x) = 1$, $B_1(x) = x - 1/2$, $B_2(x) = x^2 - x + 1/6$, $B_3(x) = x^3 - (3/2) x^2 + (1/2) x$, \ldots\ 
We are interested in their properties on the interval $[0,1]$.
We also define scaled Bernoulli polynomials for usage on arbitrary intervals $[a,b]$, with $a < b$, with equivalent properties.
Those will be denoted by $B_\tau^{[a,b]}$.

The following properties are well known, for $\tau, \tau' \in \{0,1,2,\ldots\}$,
\begin{align*}
  \int_0^1 B_\tau(x) \rd{x}
  &=
  \begin{cases}
    1, & \text{if $\tau = 0$}, \\
    0, & \text{if $\tau \ne 0$},
  \end{cases}
  &
  \frac{\rmd^{\tau'}}{\rmd x^{\tau'}} \frac{B_\tau(x)}{\tau!}
  =
  \frac{B^{(\tau')}_\tau(x)}{\tau!}
  &=
  \begin{cases}
    0,                                       & \text{if $\tau' > \tau$}, \\[2mm]
    \displaystyle
    \frac{B_{\tau-\tau'}(x)}{(\tau-\tau')!}, & \text{if $\tau' \le \tau$} .
  \end{cases}
\end{align*}
On the interval $[0,1]$ the Bernoulli polynomials have the following Fourier expansion which we will use to define the ``periodic Bernoulli polynomials'' denoted by $\tilB_\tau$ as follows
\begin{align*}
  \frac{\tilB_\tau(x)}{\tau!}
  &:=
  \frac{-1}{(2\pi\rmi)^\tau}
  \sum_{0 \ne h \in \Z} \frac{\exp(2\pi\rmi\,h x)}{h^\tau}
  &
  \begin{cases}
    \text{for $\tau=2,3,\ldots$ and $(x \bmod 1) \in [0,1]$, or}, \\
    \text{for $\tau=1$ and $(x \bmod 1) \in (0,1)$}.
  \end{cases}
\end{align*}
For $\tau=2,3,\ldots$ we have $\tilB_\tau(x) = B_\tau(x)$ when $x \in [0,1]$ and for $\tau=1$ we have $\tilB_1(x) = B_1(x)$ when $x \in (0,1)$.
From here it follows that for $\tau=1,2,3,\ldots$,
\begin{align}\label{eq:Btilde-integral}
  \int \frac{\tilB_\tau(x)}{\tau!} \rd x
  &=
  \frac{-1}{(2\pi\rmi)^{\tau+1}} \sum_{0 \ne h \in \Z} \frac{\exp(2\pi\rmi\,h x)}{h^{\tau+1}}
  =
  \frac{\tilB_{\tau+1}(x)}{(\tau+1)!}
\end{align}
and, again for $\tau=1,2,3,\ldots$,
\begin{align*}
  \int \frac{\tilB_\tau(z-y)}{\tau!} \rd y
  &=
  \frac{1}{(2\pi\rmi)^{\tau+1}} \sum_{0 \ne h \in \Z} \frac{\exp(2\pi\rmi\,h (z-y))}{h^{\tau+1}}
  =
  \frac{-\tilB_{\tau+1}(z-y)}{(\tau+1)!}
  .
\end{align*}
We also have the following property, for $\tau=0,1,2,\ldots$,
\begin{align*}
  \tilB_\tau(x-y)
  =
  (-1)^\tau \, \tilB_\tau(y-x)
  ,
\end{align*}
which follows immediately from the symmetry $B_\tau(x) = (-1)^\tau \, B_\tau(1-x)$ for $x \in [0,1]$ and $(1 - (x - y)) \bmod 1 = (y - x) \bmod 1$.

We now define scaled Bernoulli polynomials on the interval $[a,b]$, $a < b$, making sure that the same properties as for the standard Bernoulli polynomials on the interval $[0,1]$ hold.
Define, for $\tau = 0,1,2,\ldots$,
\begin{align}\label{eq:def-Bab}
  B^{[a,b]}_\tau(x)
  &:=
  (b - a)^{\tau-1} \, B_\tau((x-a)/(b-a))
  .
\end{align}
Note that $B^{[a,b]}_0(x) = 1 / (b-a)$ and $B^{[a,b]}_1(x) = B_1((x-a)/(b-a))$ such that $B^{[a,b]}_1(a) = B_1(0) = -1/2$ and $B^{[a,b]}_1(b) = B_1(1) = 1/2$. These properties will be essential, e.g., in the proof of Lemma~\ref{lem:bernpartintab}.
It can be easily verified that
\begin{align*}
  \int_a^b B^{[a,b]}_\tau(x) \rd{x}
  &=
  \begin{cases}
    1, & \text{if $\tau = 0$}, \\
    0, & \text{if $\tau \ne 0$},
  \end{cases}
\intertext{and}
  \frac{\rmd^{\tau'}}{\rmd x^{\tau'}} \frac{B^{[a,b]}_\tau(x)}{\tau!}
  =
  \frac{{B^{[a,b]}_\tau}^{(\tau')}(x)}{\tau!}
  &=
  \begin{cases}
    0,                                       & \text{if $\tau' > \tau$}, \\[2mm]
    \displaystyle
    \frac{B^{[a,b]}_{\tau-\tau'}(x)}{(\tau-\tau')!}, & \text{if $\tau' \le \tau$} .
  \end{cases}
\end{align*}
For the Fourier series on the interval $[a,b]$ we fix the basis functions, for $h \in \Z$, to be
\begin{align}\label{eq:exp-ab}
  \varphi^{[a,b]}_h(x)
  &:=
  \frac{\exp(2\pi\rmi \, h (x-a)/(b-a))}{\sqrt{b-a}}
  ,
\end{align}
such that they are orthogonal and normalized with respect to the standard $L_2$-inner product on $[a,b]$, i.e., for $h,h' \in \Z$, we have
\begin{align*}
  \int_a^b \varphi^{[a,b]}_h(x) \, \overline{\varphi^{[a,b]}_{h'}(x)} \rd x
  &=
  \begin{cases}
  1, & \text{if $h=h'$}, \\
  0, & \text{if $h\ne h'$}.
  \end{cases}
\end{align*}
We can now also define periodic Bernoulli polynomials on the interval $[a,b]$ as follows
\begin{align*}
  \frac{\tilB^{[a,b]}_\tau(x)}{\tau!}
  &:=
  \frac{-(b-a)^{\tau-1/2}}{(2\pi\rmi)^\tau}
  \sum_{0 \ne h \in \Z} \frac{\varphi^{[a,b]}_h(x)}{h^\tau}
  &
  \begin{cases}
    \text{for $\tau=2,3,\ldots$ and $x \in \mathbb{T}[a,b]$, or}, \\
    \text{for $\tau=1$ and $x \in \mathbb{T}(a,b)$},
  \end{cases}
\end{align*}
in which the notation $x \in \mathbb{T}[a,b]$ means that the value of $x$ is identified with its value on the closed torus $[a,b]$, likewise for $x \in \mathbb{T}(a,b)$ where $x$ is identified with its value on the open torus $(a,b)$.
Obviously for $\tau=2,3,\ldots$ we have $\tilB^{[a,b]}_\tau(x) = B^{[a,b]}_\tau(x)$ when $x \in [a,b]$ and for $\tau=1$ we have $\tilB^{[a,b]}_1(x) = B^{[a,b]}_1(x)$ when $x \in (a,b)$.
The periodicity implies that
\begin{align*}
  \tilB^{[a,b]}_\tau(x)
  &=
  \tilB^{[a,b]}_\tau(x + k \, (b-a))
  \qquad
  \forall x \in \R,\;
  \forall k \in \Z
  .
\end{align*}
It follows that for $\tau=1,2,3,\ldots$,
\begin{align*}
  \int \frac{\tilB^{[a,b]}_\tau(x)}{\tau!} \rd x
  &=
  \frac{-(b-a)^{\tau-1/2+1}}{(2\pi\rmi)^{\tau+1}} \sum_{0 \ne h \in \Z} \frac{\varphi^{[a,b]}_h(x)}{h^{\tau+1}}
  =
  \frac{\tilB^{[a,b]}_{\tau+1}(x)}{(\tau+1)!}
\end{align*}
and, again for $\tau=1,2,3,\ldots$,
\begin{align*}
  \int \frac{\tilB^{[a,b]}_\tau(z-y)}{\tau!} \rd y
  &=
  \frac{-(b-a)^{\tau-1/2+1}}{(2\pi\rmi)^{\tau+1}} \sum_{0 \ne h \in \Z} \frac{\varphi^{[a,b]}_h(z-y)}{h^{\tau+1}}
  =
  \frac{-\tilB^{[a,b]}_{\tau+1}(z-y)}{(\tau+1)!}
  .
\end{align*}
We also have the following property, for $\tau=0,1,2,\ldots$,
\begin{align*}
  \tilB^{[a,b]}_\tau(x-y+a)
  &=
  (-1)^\tau \, \tilB^{[a,b]}_\tau(b-(x-y)-(b-a))
  =
  (-1)^\tau \, \tilB^{[a,b]}_\tau(y-x+a)
  ,
\end{align*}
which followed immediately from the symmetry $B^{[a,b]}_\tau(a+t) = (-1)^\tau \, B^{[a,b]}_\tau(b-t)$ for $t \in [0,b-a]$ and the periodicity with period $(b-a)$.

The maximal magnitude of the scaled Bernoulli polynomials on the interval $[a,b]$ can be bounded by
\begin{align*}
  \forall x \in [a,b],
  \; \forall \tau \ge 2
  :
  \qquad
  |B^{[a,b]}_\tau(x)|
  &\le
  (b-a)^{\tau-1} \frac{2 \, \tau!}{(2\pi)^\tau}
  \begin{cases}
    \zeta(\tau), & \text{if $\tau \equiv 0 \pmod{2}$}, \\
    1          , & \text{otherwise},
  \end{cases}
\end{align*}
where $\zeta(\tau) := \sum_{k \ge 1} k^{-\tau}$ is the Riemann zeta function.
For the derivation of this bound we refer to~\cite{L1940}. 
Note that this means that, for any interval $[a, b]$ with $a < b$,
\begin{align}\label{eq:BerBound}
  \forall x \in [a,b],
  \; \forall \tau \ge 1
  :
  \qquad
  \frac{|B^{[a,b]}_\tau(x)|}{\tau!}
  \le
  \frac{(b-a)^{\tau-1}}{2}
  .
\end{align}

\subsection{RKHSs on the unit cube}

With $\calH(K)$ we denote the reproducing kernel Hilbert space which corresponds to a reproducing kernel $K: [0,1]^d \times [0,1]^d \to \R$.
By $\|f\|_K$ and $\langle f , g \rangle _K$ we mean the corresponding norm of $f \in \calH(K)$ and inner product of $f$ and $g$ for $f, g \in \calH(K)$. 
We always have $\|f\|_K^2 = \langle f , f \rangle _K$.
The key property of a reproducing kernel Hilbert space is that the kernel reproduces the function values by means of the inner product, i.e.,
\begin{align*}
  \langle f, K(\cdot, \bsy) \rangle_K
  &=
  f(\bsy)
  \qquad
  \forall \bsy \in [0,1]^d
  \text{ and }
  \forall f \in \calH(K)
  .
\end{align*}

We first define our non-periodic function space on the unit cube.
For the derivatives of a function $f$ we use the notation
\[
 f^{(\bstau)}(\bsx)
 =
 \frac{\partial^{\tau_1+\cdots+\tau_d} f(\bsx)}{\partial x_1^{\tau_1} \cdots \partial x_d^{\tau_d}}
 .
\]
The spaces on the unit cube which we introduce here are well known and can also be found in, e.g., \cite{NW2008}. We introduce them to be able to explicitly see the parallels and differences with the spaces on a box which we will define in the next subsection.
We remind the reader that we use the shorthands $[\bszero_\setw,\bsone_\setw] = \prod_{j \in \setw} [0,1] = [0,1]^{|\setw|}$ and $[\bszero_{-\setw},\bsone_{-\setw}] = \prod_{j \notin \setw} [0,1] = [0,1]^{d-|\setw|}$.

\begin{definition}\label{def:Sob}
For $\alpha \in \N$ the space $\calH(\KSob_\alpha)$ is an unanchored Sobolev space on the unit cube $[0,1]^d$, which is a reproducing kernel Hilbert space with inner product
\begin{align*}
    \langle f , g \rangle _{\KSob_{\alpha,d}}
    :=
    \sum_{ \substack{ \bstau \in \{0,\ldots,\alpha \}^d \\ \setw := \{j: \tau_j=\alpha \}} }  
    \int_{[\bszero_\setw,\bsone_\setw]} 
    \left(\int_{[\bszero_{-\setw},\bsone_{-\setw}]} f^{(\bstau)} (\bsx) \rd \bsx_{-\setw} 
    \right)
    \left(\int_{[\bszero_{-\setw},\bsone_{-\setw}]} g^{(\bstau)} (\bsx) \rd \bsx_{-\setw} 
    \right)
    \rd \bsx_\setw
    ,
\end{align*}
and reproducing kernel
\begin{align*}
 \KSob_{\alpha,d} (\bsx,\bsy)
 &:=
 \prod_{j=1}^d \left( 1 + \sum_{\tau_j=1}^\alpha \frac{B_{\tau_j}(x_j)}{\tau_j!} \, \frac{B_{\tau_j}(y_j)}{\tau_j!} + (-1)^{\alpha+1} \frac{\tilB_{2\alpha}(x_j-y_j)}{(2\alpha)!} \right)
 .
\end{align*}
\end{definition}

Functions in this space have mixed partial derivatives up to order $\alpha-1$ in each variable which are absolutely continuous and of which the highest derivatives of order $\alpha$ are $L_2$ integrable.
We will consider next a function space with the same properties, but in addition ask that the functions are ``periodic''.

All functions in the periodic space will allow to be expressed as absolutely converging Fourier series (which implies pointwise equality of the Fourier series)
\begin{align*}
  f(\bsx)
  &=
  \sum_{\bsh \in \Z^d} \widehat{f}(\bsh) \, \exp(2\pi\rmi \, \bsh \cdot \bsx)
  ,
  &
  \widehat{f}(\bsh)
  &:=
  \int_{[0,1]^d} f(\bsx) \, \exp(-2\pi\rmi \, \bsh \cdot \bsx) \rd{\bsx}
  .
\end{align*}
We define the periodic function space in such a way that the norms of the periodic functions are equal in both spaces for the same integer smoothness.

\begin{definition}\label{def:Kor}
  For $\alpha > 1/2$ the space $\calH(\KKor_{\alpha,d})$ is a Korobov space on the unit cube $[0,1]^d$, which is a reproducing kernel Hilbert space with inner product 
  \begin{align*}
    \langle f , g \rangle_{\KKor_{\alpha,d}}
    &:=
    \sum_{\bsh\in\Z^d} \widehat{f}(\bsh) \, \overline{\widehat{g}(\bsh)} \,\left[r_{\alpha,d}(\bsh)\right]^2
    ,
  \end{align*}
  and reproducing kernel
  \begin{align}\label{eq:KKor}
    \KKor_{\alpha,d} (\bsx,\bsy)
    &:=
    \sum_{\bsh \in \Z^d} \frac{\exp(2\pi\rmi \, \bsh\cdot(\bsx-\bsy))}{\left[r_{\alpha,d}(\bsh)\right]^2}
    =
    \prod_{j=1}^d \left( 1 +  (-1)^{\alpha+1} \frac{\tilB_{2\alpha}(x_j-y_j)}{(2\alpha)!} \right)
    ,
  \end{align}
  where the second form holds when $\alpha \in \N$, and with
\begin{align*}
  r_{\alpha,d}(\bsh)
  &:=
  \prod_{j=1}^d r_\alpha(h_j)
  ,
  &
  r_\alpha(h_j)
  &:=
  \begin{cases}
    1,             & \text{for $h_j = 0$},   \\
    |2\pi \, h_j|^\alpha, & \text{for $h_j \ne 0$}.
  \end{cases}
\end{align*}  
\end{definition}

For $\alpha \in \N$ the normalization in $r_{\alpha,d}$ is chosen in such a way that the inner product and norm coincides with that of the unanchored Sobolev space $\calH(\KSob_{\alpha,d})$ of smoothness $\alpha$ for functions in $\calH(\KKor_{\alpha,d})$.
Functions in the Korobov space $\calH(\KKor_{\alpha,d})$ have mixed partial derivatives up to order $\alpha-1$ in each variable which are absolutely continuous, and of which derivatives of order $\alpha$ are $L_2$ integrable, and, additionally, the derivatives up to order $\alpha-1$ have matching values on the boundaries, i.e., periodic boundary conditions,
\begin{align*}
  \forall \bstau \in \{0,\ldots,\alpha-1\}^d, \forall j=1,\ldots,d :
  \bigl. f^{(\bstau)}(\bsx) \bigr|_{x_j=1}
  =
  \bigl. f^{(\bstau)}(\bsx) \bigr|_{x_j=0}
  ,
\end{align*}
and, as a consequence,
\begin{multline}\label{eq:periodicity-integrals}
  \forall \bstau \in \{1,\ldots,\alpha\}^d,
  \forall \setv \subseteq \{1,\ldots,d\} \setminus \{ j : \tau_j = \alpha \} , \setv \ne \emptyset ,
  \forall j \in \setv :
  \\
  \int_{[0,1]^{|\setv|}} f^{(\bstau)}(\bsx) \rd\bsx_\setv
  =
  \int_{[0,1]^{|\setv|-1}} \left[
  \bigl. f^{(\bstau-1)}(\bsx) \bigr|_{x_j=1}
  -
  \bigl. f^{(\bstau-1)}(\bsx) \bigr|_{x_j=0}
  \right] \rd\bsx_{\setv \setminus \{j\}}
  =
  0
  .
\end{multline}
For $\bstau \in \{0,\ldots,\alpha\}^d$ the mixed partial derivatives can be written in terms of Fourier series with modified coefficients:
\begin{align*}
  f^{(\bstau)}(\bsx)
  &=
  \sum_{\substack{\bsh \in \Z^d \\ h_j \ne 0 \text{ for } \tau_j \ne 0}} \left[ \prod_{j=1}^d (2\pi\rmi \, h_j)^{\tau_j} \right] \, \widehat{f}(\bsh) \, \exp(2\pi\rmi \, \bsh \cdot \bsx)
  &&
  (\text{a.e.\ when any } \tau_j = \alpha)
  ,
\end{align*}
which is $L_2$ integrable since $\|f\|_{\KKor_{\alpha,d}}^2 < \infty$ for $f \in \calH(\KKor_{\alpha,d})$.
We do remark that as soon as any of the $\tau_j = \alpha$ then the above equality holds only in the almost everywhere sense and we then have a so-called weak derivative. This is of no concern to us as we do not need point evaluation of the highest order derivatives in any of our derivations.
It is obvious that the Korobov space is a subspace of the Sobolev space and using the previous periodicity properties in the integrals in the inner product of the Sobolev space (for the $\tau_j \in \{1,\ldots,\alpha-1\}$) it can be seen that for $f, g \in \calH(\KKor_{\alpha,d})$ the inner product for $\calH(\KKor_{\alpha,d})$ with $\alpha \in \N$ can be written in two ways: using Fourier series or using derivatives
\begin{align*}
 \langle f , g \rangle_{\KKor_{\alpha,d}}
 &=
 \sum_{\bsh\in\Z^d} \widehat{f}(\bsh) \; \overline{\widehat{g}(\bsh)}  \; \left[r_{\alpha,d}(\bsh)\right]^2
 \\
  &=
   \sum_{ \substack{ \bstau \in \{0,\alpha \}^d \\ \setw := \{j: \tau_j=\alpha \}} }  
    \int_{[\bszero_\setw,\bsone_\setw]} 
    \left(\int_{[\bszero_{-\setw},\bsone_{-\setw}]} f^{(\bstau)} (\bsx) \rd \bsx_{-\setw} 
    \right)
    \left(\int_{[\bszero_{-\setw},\bsone_{-\setw}]} g^{(\bstau)} (\bsx) \rd \bsx_{-\setw} 
    \right)
    \rd \bsx_\setw
  \\&\hphantom{:}=
  \langle f , g \rangle_{\KSob_{\alpha,d}}
  ,
\end{align*}
where the sum in the second line is now only over $\bstau$ with components $0$ or $\alpha$ because of the periodicity of the functions $f$ and $g$.

\subsection{RKHSs on a box}

We now want to define a function space for non-periodic functions on a box $[\bsa,\bsb]$ where the inner product is a direct extension of the one over the unit cube from the previous section.

\begin{definition}\label{def:Sob-ab}
For $\alpha \in \N$ the space $\calH(\KSobab_\alpha)$ is an unanchored Sobolev space on the box $[\bsa,\bsb]$, which is a reproducing kernel Hilbert space with inner product
\begin{multline}\label{eq:Sob-ab-innerproduct}
    \langle f , g \rangle _{\KSobab_{\alpha,d}}
    \\:=
    \sum_{ \substack{ \bstau \in \{0,\ldots,\alpha \}^d \\ \setw := \{j: \tau_j=\alpha \}} }  
    \int_{[\bsa_\setw,\bsb_\setw]} 
    \left(\int_{[\bsa_{-\setw},\bsb_{-\setw}]} f^{(\bstau)} (\bsx) \rd \bsx_{-\setw} 
    \right)
    \left(\int_{[\bsa_{-\setw},\bsb_{-\setw}]} g^{(\bstau)} (\bsx) \rd \bsx_{-\setw} 
    \right)
    \rd \bsx_\setw,
\end{multline}
and reproducing kernel
\begin{multline}\label{eq:Sob-ab-kernel}
 \KSobab_{\alpha,d}(\bsx,\bsy)
 :=
 \prod_{j=1}^d \Bigg(
   \frac{1}{(b_j-a_j)^2}
   +
   \sum_{\tau_j=1}^{\alpha-1} \frac{B^{[a_j, b_j]}_{\tau_j}(x_j)}{\tau_j!} \frac{B^{[a_j, b_j]}_{\tau_j}(y_j)}{\tau_j!} 
   \\+
   (b_j-a_j)\frac{B^{[a_j,b_j]}_\alpha(x_j)}{\alpha!} \frac{B^{[a_j,b_j]}_\alpha (y_j)}{\alpha!}
   +
   (-1)^{\alpha+1} \frac{\tilB^{[a_j,b_j]}_{2\alpha}(x_j-y_j+a_j)}{(2\alpha)!}
 \Bigg)
 .
\end{multline}
\end{definition}

We will show in Proposition~\ref{prop:kernelSobab} that $\KSobab_{\alpha,d}$ is indeed the kernel corresponding to the inner product above.

For the periodic space we define scaled periodic basis functions which are just the tensor product of the one-dimensional Fourier basis on $[a,b]$ given in~\eqref{eq:exp-ab}, that is, for $\bsh \in \Z^d$, we define
\begin{align*}
  \varphi_\bsh^{[\bsa,\bsb]}(\bsx) 
  &:=
  \prod_{j=1}^d \varphi_{h_j}^{[a_j,b_j]}(x_j)
  = 
  \prod_{j=1}^d \frac{\exp(2\pi \rmi \, h_j (x_j-a_j) / (b_j-a_j))}{ \sqrt{b_j-a_j}}
  .
\end{align*}
They form an orthogonal and normalized set of basis functions against the standard $L_2$ inner product on the box $[\bsa,\bsb]$.
Therefore, our periodic functions are expressed as absolutely converging Fourier series with respect to the basis $\varphi_\bsh^{[\bsa,\bsb]}$ as follows
\begin{align*}
  f(\bsx)
  &=
  \sum_{\bsh \in \Z^d}  \widehat{f}^{[\bsa,\bsb]} \, \varphi_\bsh^{[\bsa,\bsb]}(\bsx)
  ,
  &
  \widehat{f}^{[\bsa,\bsb]}(\bsh)
  &:=
  \int_{[\bsa,\bsb]} f(\bsx) \, \overline{\varphi_\bsh^{[\bsa,\bsb]}(\bsx)} \rd \bsx
  .
\end{align*}

\begin{definition}\label{def:Kor-ab}
  For $\alpha > 1/2$ the space $\calH(\KKorab_{\alpha,d})$ is a Korobov space on the box $[\bsa,\bsb]$, which is a reproducing kernel Hilbert space with inner product
\begin{align*}
 \langle f , g \rangle_{\KKorab_{\alpha,d}}
 &:=
 \sum_{\bsh\in\Z^d} \widehat{f}^{[\bsa,\bsb]}(\bsh)\;  \overline{ \widehat{g}^{[\bsa,\bsb]}(\bsh)} \; \left[r^{[\bsa,\bsb]}_{\alpha,d}(\bsh)\right]^{2}
 ,
\end{align*}
and reproducing kernel
\begin{align*}
  \KKorab_{\alpha,d} (\bsx,\bsy)
  &:=
  \sum_{\bsh\in\Z^d} [r_{\alpha,d}^{[\bsa,\bsb]}(\bsh)]^{-2} \, \varphi_\bsh^{[\bsa,\bsb]}(\bsx) \, \overline{\varphi_\bsh^{[\bsa,\bsb]}(\bsy)}
  \\
  &\hphantom{:}=
  \prod_{j=1}^d \left( \frac{1}{(b_j-a_j)^2} + (-1)^{\alpha+1} \frac{\tilB^{[a_j,b_j]}_{2\alpha}(x_j-y_j+a_j)}{(2\alpha)!} \right)
  ,
\end{align*}
where the second form holds when $\alpha \in \N$, and with
\begin{align*}
  r_{\alpha,d}^{[\bsa,\bsb]}(\bsh)
  &:=
  \prod_{j=1}^d r_\alpha^{[a_j,b_j]}(h_j)
  ,
  &
  r_\alpha^{[a_j,b_j]}(h_j)
  &:=
  \begin{cases}
    \sqrt{b_j - a_j} , & \text{for $h_j = 0$}, \\[2mm]
    \displaystyle \frac{|2\pi\,h_j|^\alpha}{(b_j-a_j)^\alpha} , & \text{for $h_j \ne 0$} .
  \end{cases}
\end{align*}
\end{definition}

The choice of $r_\alpha^{[a,b]}$ made in the definition is chosen such that functions in the Sobolev space of smoothness $\alpha \in \N$ which are periodic will have the same norm in this Korobov space.
This is shown in the next lemma.

\begin{lemma}
  For $\alpha \in \N$ and $f \in \calH(\KKorab_{\alpha,d})$ we have $\calH(\KKorab_{\alpha,d}) \subset \calH(\KSobab_{\alpha,d})$ and
  \[
    \|f\|_{\KKorab_{\alpha,d}} = \|f\|_{\KSobab_{\alpha,d}}
    .
  \]
\end{lemma}
\begin{proof}
As is the case in the unit cube we only need to consider partial derivatives of order $0$ and $\alpha$ in the norm of the Sobolev space when the function is periodic as all the interior integrals vanish due to the periodicity, see~\eqref{eq:periodicity-integrals}.
For $f,g \in \calH(\KKorab_{\alpha,d})$ we can thus write
\begin{align*}
  \langle f, g \rangle_{\KSobab_{\alpha,d}}
  &=
  \hspace{-5mm}  \sum_{ \substack{ \bstau \in \{0 , \alpha \}^d \\ \setw := \{j: \tau_j=\alpha \}} }  
    \int_{[\bsa_\setw,\bsb_\setw]} 
    \left(\int_{[\bsa_{-\setw},\bsb_{-\setw}]} f^{(\bstau)} (\bsx) \rd \bsx_{-\setw} 
    \right)
    \left(\int_{[\bsa_{-\setw},\bsb_{-\setw}]} g^{(\bstau)} (\bsx) \rd \bsx_{-\setw} 
    \right)
    \rd \bsx_\setw
    .
\end{align*}
Therefore, for $f \in \calH(\KKorab_{\alpha,d})$ and $\bstau \in \{0,\alpha\}^d$ we have, with $\setw = \{j : \tau_j = \alpha\}$,
\begin{align*}
  &
  \int_{[\bsa_\setw,\bsb_\setw]} \left(
    \int_{[\bsa_{-\setw},\bsb_{-\setw}]} f^{(\bstau)}(\bsx) \rd{\bsx_{-\setw}}
  \right)^2 \rd{\bsx_\setw}
  \\
  &\quad=
  \int_{[\bsa_\setw,\bsb_\setw]} \left(
  \sum_{\substack{\bsh \in \Z^d \\ h_j \ne 0 \text{ for } j \in \setw}} \hspace{-5mm} \widehat{f}^{[\bsa,\bsb]}(\bsh) \, 
    \left[
      \prod_{j \in \setw} \frac{(2\pi\rmi\,h_j)^\alpha}{(b_j-a_j)^\alpha} \, \varphi^{[a_j,b_j]}_{h_j}(x_j)
    \right]
    \,
    \left[ \prod_{j \notin \setw} 
      \underbrace{\int_{a_j}^{b_j} \varphi^{[a_j,b_j]}_{h_j}(x_j) \rd{x_j}}_{= 0 \text{ if } h_j \ne 0}
    \right]
  \right)^2 \hspace{-3mm} \rd{\bsx_\setw} 
  \\
  &\quad=
  \int_{[\bsa_\setw,\bsb_\setw]} \left(
  \sum_{\substack{\bsh \in \Z^d \\ h_j \ne 0 \text{ for } j \in \setw \\ h_j = 0 \text{ for } j \notin \setw}} \hspace{-5mm}  \widehat{f}^{[\bsa,\bsb]}(\bsh) \, 
    \left[
      \prod_{j \in \setw} \frac{(2\pi\rmi\,h_j)^\alpha}{(b_j-a_j)^\alpha} \, \varphi^{[a_j,b_j]}_{h_j}(x_j)
    \right]
    \,
    \left[ \prod_{j \notin \setw} 
      (b_j-a_j)^{1-1/2}
    \right]
  \right)^2 \hspace{-3mm} \rd{\bsx_\setw} 
  \\
  &\quad=
  \left(\prod_{j\notin\setw} \sqrt{b_j-a_j}\right)^2
  \\ &\qquad\qquad
  \int_{[\bsa_\setw,\bsb_\setw]} \left(
  \sum_{\bsh_\setw \in (\Z \setminus\{0\})^{|\setw|}}
    \widehat{f}^{[\bsa,\bsb]}(\bsh_\setw;\bszero_{-\setw}) \, 
    \left[
      \prod_{j \in \setw} \frac{(2\pi\rmi\,h_j)^\alpha}{(b_j-a_j)^\alpha} \, \varphi^{[a_j,b_j]}_{h_j}(x_j)
    \right]
  \right)^2 \rd{\bsx_\setw}
  \\
  &\quad=
  \sum_{\substack{{\bsh_\setw \in (\Z\setminus\{0\})^{|\setw|} }}}
    \left|\widehat{f}^{[\bsa,\bsb]}(\bsh)\right|^2 \, 
    [r_{\alpha,d}^{[\bsa,\bsb]}(\bsh)]^2
  ,
\end{align*}
and thus for $f \in \calH(\KKorab_{\alpha,d})$
\begin{align*}
  \|f\|_{\KKorab_{\alpha,d}}^2
  &=
  \sum_{\bsh \in \Z^d} |\widehat{f}^{[\bsa,\bsb]}(\bsh)|^2 \,
  [r_{\alpha,d}^{[\bsa,\bsb]}(\bsh)]^2
  \\
  &=
  \sum_{\setw \subseteq \{1,\ldots,d\}} 
  \sum_{{\bsh_\setw \in (\Z\setminus\{0\})^{|\setw|} }}
    \left|\widehat{f}^{[\bsa,\bsb]}(\bsh)\right|^2 \, 
    [r_{\alpha,d}^{[\bsa,\bsb]}(\bsh)]^2
  =
  \|f\|_{\KSobab_{\alpha,d}}^2
  .
  \qedhere
\end{align*}
\end{proof}

Since we will apply a scaled lattice rule to a function in $\calH(\KKorab_{\alpha,d})$ we show how to bound the error in terms of the worst-case error of the same lattice rule on the unit cube.

\begin{proposition}\label{prop:lattice-rule-error-ab}
  For $\alpha \in \N$ and $f \in \calH(\KKorab_{\alpha,d}) \subset \calH(\KSobab_{\alpha,d})$ the error of using an $n$-point lattice rule with generating vector $\bsz \in \Z^d$ and scaled nodes $\bsp_i^{[\bsa,\bsb]}$ can be bounded as
  \begin{multline*}
    \left|
    \frac{\prod_{j=1}^d (b_j-a_j)}{n} \sum_{i=1}^n f(\bsp_i^{[\bsa,\bsb]})
    -
    \int_{[\bsa,\bsb]} f(\bsx) \rd \bsx
    \right|
    \\
    \le
    \|f\|_{\KKorab_{\alpha,d}}
    \,
    \left( \sum_{\bszero \ne \bsh \in \Lambda^\perp(\bsz,n)} [r_{\alpha,d}^{[\bsa,\bsb]}(\bsh)]^{-2} \prod_{j=1}^d (b_j - a_j) \right)^{1/2}
    \\
    \le
    \|f\|_{\KKorab_{\alpha,d}}
    \,
    \left( \prod_{j=1}^d \max(1, b_j-a_j)^{\alpha+1/2} \right)
    \left( \sum_{\bszero \ne \bsh \in \Lambda^\perp(\bsz,n)} [r_{\alpha,d}(\bsh)]^{-2} \right)^{1/2}
    ,
  \end{multline*}
  where $\Lambda^\perp(\bsz,n) := \{ \bsh \in \Z^d : \bsh \cdot \bsz \equiv 0 \pmod{n} \}$ is the dual of the lattice.
  The norm $\|f\|_{\KKorab_{\alpha,d}}$ can be replaced by the norm in the associated Sobolev space since $\|f\|_{\KKorab_{\alpha,d}} = \|f\|_{\KSobab_{\alpha,d}}$ for the spaces here defined with $f \in \calH(\KKorab_{\alpha,d}) \subset \calH(\KSobab_{\alpha,d})$.
\end{proposition}
\begin{proof}
  First we note that
  \begin{align*}
     \widehat{f}^{[\bsa,\bsb]}(\bszero) \prod_{j=1}^d \sqrt{b_j-a_j}
     &=
     \int_{[\bsa,\bsb]} f(\bsx) \rd\bsx
     .
  \end{align*}
  Therefore
  \begin{align*}
    &
    \frac{\prod_{j=1}^d (b_j-a_j)}{n} \sum_{i=1}^n f(\bsp_i^{[\bsa,\bsb]})
    -
    \int_{[\bsa,\bsb]} f(\bsx) \rd \bsx
    \\
    &\qquad=
    \sum_{\bsh \in \Z^d}
    \widehat{f}^{[\bsa,\bsb]}(\bsh)
    \frac{\prod_{j=1}^d \sqrt{b_j-a_j}}{n}
    \sum_{i=1}^n \exp(2\pi\rmi \, \bsh \cdot \bsz \, i / n)
    -
    \widehat{f}^{[\bsa,\bsb]}(\bszero) \prod_{j=1}^d \sqrt{b_j-a_j}
    \\
    &\qquad=
    \left[ \prod_{j=1}^d \sqrt{b_j-a_j} \right]
    \sum_{\bszero \ne \bsh \in \Lambda^\perp(\bsz,n)}
    \widehat{f}^{[\bsa,\bsb]}(\bsh)
    ,
  \end{align*}
  where we made use of the character property, for $t \in \Z$,
  \begin{align*}
    \frac1n \sum_{i=1}^n \exp(2\pi\rmi \, t \, i/n)
    &=
    \begin{cases}
      1, & \text{if $t \equiv 0 \pmod{n}$}, \\
      0, & \text{otherwise},
    \end{cases}
  \end{align*}
  with $t = \bsh \cdot \bsz$.
  The result now follows by multiplying and dividing with $r_{\alpha,d}^{[\bsa,\bsb]}(\bsh)$ and applying the Cauchy--Schwarz inequality to arrive at the norm of $f$ in the Korobov space $\calH(\KKorab_{\alpha,d})$.
  The last inequality in the claim follows from using the definitions of $r_{\alpha,d}^{[\bsa,\bsb]}(\bsh)$ and $r_{\alpha,d}(\bsh)$.
\end{proof}

\section{The Bernoulli polynomial method as a projection}\label{sec:BPM-projection}

The Bernoulli polynomial method was introduced by Korobov \cite{K1963} and  Zaremba \cite{Z1972}.
We first revisit the method on the unit cube and show it can be interpreted as an orthogonal projection and then we generalize this analysis to arbitrary boxes $[\bsa,\bsb]$.
In Section~\ref{sec:error-analysis}, this projection will be used for measuring the non-periodicity of a function.

\subsection{The Bernoulli polynomial method on the unit cube}

We first review the Bernoulli polynomial method on the unit cube and conclude that this method coincides with an orthogonal projection from an unanchored Sobolev space $\calH(\KSob_{\alpha,d})$ to a Korobov space $\calH(\KKor_{\alpha,d})$ where those two spaces have exactly the same norms for a function $f\in\calH(\KKor_{\alpha,d}) \subset \calH(\KSob_{\alpha,d})$.

In \cite{K1963}, the Bernoulli polynomial method was originally introduced. In the present paper, to introduce the method we follow the definition by \cite{BH1992,SJ1994} for the sake of notational simplicity.
Define the recursion 
\begin{align}\label{eq:BPM-recursive}
\begin{cases}
  F_{[0]}(\bsx)
  &:=\;
  f(\bsx)
  ,
  \\
  F_{[j]}(\bsx)
  &:=\;\displaystyle
  F_{[j-1]}(\bsx)
  -
  \sum_{\tau_j=1}^\alpha
  \frac{B_{\tau_j}(x_j)}{\tau_j!}
  \left[
    \frac{\partial^{\tau_j-1} F_{[j-1]}(\bsx)}{\partial x_j^{\tau_j-1}}
  \right]_{x_j=0}^1
  \\
  &\hphantom{:}=\;\displaystyle
  F_{[j-1]}(\bsx)
  -
  \sum_{\tau_j=1}^\alpha
  \frac{B_{\tau_j}(x_j)}{\tau_j!}
  \int_0^{1\vphantom{A^{A^A}}}
    \frac{\partial^{\tau_j} F_{[j-1]}(\bsx)}{\partial x_j^{\tau_j}}
  \rd{x_j}
  \qquad \text{for } j = 1, \ldots, d
  ,
\end{cases}
\end{align}
where we used that in our setting all derivatives up to order $\alpha-1$ are absolutely continuous such that we can replace the difference by an integral.
The final function
\[
  F := F_{[d]}
\]
now has the property of the Korobov space in that it satisfies that all derivatives up to order $\alpha-1$, in all dimensions, have matching boundary values.
Furthermore, since the Bernoulli polynomials integrate to zero, the integral of $F$ is equal to the integral of the original function $f$.

From the recursive definition above we can reach the following expression \cite[p.~61]{Z1972}, see Lemma~\ref{lem:BPM} below,
\begin{align}
  \notag
  F(\bsx)
  &:=
  f(\bsx)
  +
      \sum_{\substack{\bszero \ne \bstau \in \{0,\ldots,\alpha\}^d \\ \setu := \supp(\bstau)}}
      (-1)^{|\setu|}
      \left[ \prod_{j \in \setu} \frac{B_{\tau_j}(x_j)}{\tau_j!} \right]
        \sum_{\setw\subseteq\setu} (-1)^{|\setw|}
          f^{(\bstau_\setu-\bsone_\setu, \bszero_{-\setu})}(\bszero_\setw,\bsone_{\setu\setminus\setw}, \bsx_{-\setu})
  \\\label{eq:BPM}
  &\hphantom{:}=
  f(\bsx)
  +
  \sum_{\substack{\bszero \ne \bstau \in \{0,\ldots,\alpha\}^d \\ \setu := \supp(\bstau)}}
      (-1)^{|\setu|}
      \left[ \prod_{j \in \setu} \frac{B_{\tau_j}(x_j)}{\tau_j!} \right]
      \int_{[0,1]^{|\setu|}} f^{(\bstau)}(\bsx) \rd\bsx_{\setu}
  .
\end{align}
We remark that the product over $j \in \setu$ could be replaced by the full product over $j=1,\ldots,d$ since $B_0(x)/0! = 1$, however, when we extend this method in the next section to arbitrary boxes $[\bsa,\bsb]$ this property is not true anymore since $B_0^{[a,b]}(x) = 1/(b-a)$ by our definition~\eqref{eq:def-Bab}. We will therefore take care to write the expressions such that their generalization to arbitrary boxes stays transparent.
We also like to point out that often we will have products $\prod_{j\in\setu}$ where $\tau_j$ could be zero for $j\in\setu$, this is unlike normal appearances in QMC literature.
In the proof of the next lemma we demonstrate that~\eqref{eq:BPM-recursive} and~\eqref{eq:BPM} are the same by an inductive proof, as this formal method returns later in the proof of Lemma~\ref{lem:bernpartint} and Proposition~\ref{prop:representations}.
For conciseness we use the integral form~\eqref{eq:BPM} which holds since in our setting all derivatives up to order $\alpha-1$ are absolutely continuous.

\begin{lemma}\label{lem:BPM}
  The recursive definition~\eqref{eq:BPM-recursive} and~\eqref{eq:BPM} are the same.
\end{lemma}
\begin{proof}
  We prove this inductively.
  We start with the base case and obtain immediately from~\eqref{eq:BPM-recursive} that
  \begin{align*}
    F_{[1]}(\bsx)
    &=
    \sum_{\setu \subseteq \{1\}} (-1)^{|\setu|}
    \sum_{\bstau_\setu \in \{1,\ldots,\alpha\}^{|\setu|}}
    \left[ \prod_{j\in\setu} \frac{B_{\tau_j}(x_j)}{\tau_j!} \right]
    \int_{[0,1]^{|\setu|}} f^{(\tau_1,0,\ldots,0)}(\bsx) \rd{\bsx_\setu}
    ,
  \end{align*}
  since for $\setu = \emptyset$ we have $\tau_1 = 0$ and obtain $f(\bsx) = F_{[0]}(\bsx)$, while for $\setu = \{1\}$ we obtain the sum over $\tau_1 = \tau_j$ from~\eqref{eq:BPM-recursive} with $F_{[0]}(\bsx) = f(\bsx)$.
  Now assume the first $j-1$ dimensions are handled and $F_{[j-1]}$ can be written as follows (by the induction hypothesis)
  \begin{align*}
    F_{[j-1]}(\bsx)
    &=
    \sum_{\setu \subseteq \{1,\ldots,j-1\}}
    (-1)^{|\setu|}
    \sum_{\bstau_\setu \in \{1,\ldots,\alpha\}^{|\setu|}}
    \left[ \prod_{j\in\setu} \frac{B_{\tau_j}(x_j)}{\tau_j!} \right]
    \int_{[0,1]^{|\setu|}} f^{(\tau_1,\ldots,\tau_{j-1},0,\ldots,0)}(\bsx) \rd{\bsx_\setu}
    ,
  \end{align*}
  then, by the recursion formula and the same reasoning as for the expression of the base case,
  \begin{align*}
    F_{[j]}(\bsx)
    &=
    \sum_{\setv \subseteq \{j\}}
    (-1)^{|\setv|}
    \sum_{\bstau_\setv \in \{1,\ldots,\alpha\}^{|\setv|}}
    \sum_{\setu \subseteq \{1,\ldots,j-1\}}
    (-1)^{|\setu|}
    \sum_{\bstau_\setu \in \{1,\ldots,\alpha\}^{|\setu|}}
    \\
    &\qquad
    \left[ \prod_{j\in\setv} \frac{B_{\tau_j}(x_j)}{\tau_j!} \right]
    \left[ \prod_{j\in\setu} \frac{B_{\tau_j}(x_j)}{\tau_j!} \right]
    \int_{[0,1]^{|\setv|}}
    \int_{[0,1]^{|\setu|}} f^{(\tau_1,\ldots,\tau_{j-1},\tau_j,0,\ldots,0)}(\bsx) \rd{\bsx_\setu} \rd{\bsx_\setv}
    .
  \end{align*}
  From this expression it is clear that we can merge the disjunctive sets $\setu$ and $\setv$ to obtain
  \begin{align*}
   F_{[j]}(\bsx)
    &=
    \sum_{\setu \subseteq \{1,\ldots,j\}}
    (-1)^{|\setu|}
    \sum_{\bstau_\setu \in \{1,\ldots,\alpha\}^{|\setu|}}
    \left[ \prod_{j\in\setu} \frac{B_{\tau_j}(x_j)}{\tau_j!} \right]
    \int_{[0,1]^{|\setu|}} f^{(\tau_1,\ldots,\tau_j,0,\ldots,0)}(\bsx) \rd{\bsx_\setu}
    .
  \end{align*}
  The expression for $F_{[d]} = F$ can then be rewritten in the form~\eqref{eq:BPM} by pulling out $\setu = \emptyset$ to obtain $f(\bsx)$ and joining the sums over $\setu$ and $\bstau_\setu$.
\end{proof}

Our main building block is the following lemma which connects the Bernoulli polynomial method~\eqref{eq:BPM} and the orthogonal projection shown in Proposition~\ref{prop:projection}.

\begin{lemma}\label{lem:bernpartint}
  For a function $f$ on the unit cube having mixed partial derivatives up to order $\alpha-1$ in each variable which are absolutely continuous and of which the derivatives with the highest order $\alpha$ are integrable, the following holds for all $\bsx \in [0,1]^d$:
\begin{multline*}
  \sum_{\setu \subseteq \{1,\ldots,d\}}
  (-1)^{d-|\setu|}
  \sum_{\substack{\bstau_\setu \in \{0,\ldots,\alpha \}^{|\setu|} \\ \tau_j = 0 \text{ for } j \notin \setu}}
  \left[ \prod_{j\in\setu} \frac{B_{\tau_j}(x_j)}{\tau_j!} \right]
  \int_{[0,1]^{|\setu|}} f^{(\bstau)}(\bsx) \rd \bsx_{\setu}
  \\
  =
  \int_{[0,1]^d}
  \left[ \prod_{j=1}^d \frac{\tilB_\alpha(x_j-y_j)}{\alpha!} \right]
  f^{(\alpha,\ldots,\alpha)}(\bsy) \rd \bsy
  ,
\end{multline*}
where the integral over the empty domain $[0,1]^{|\setu|}$ when $\setu = \emptyset$ is the identity operation.
In this way the term for $\setu=\emptyset$ in the left hand side corresponds to $(-1)^d f(\bsx)$.
\end{lemma}
\begin{proof}
First we prove the one-dimensional case.
For $\tau = 1, \ldots, \alpha-1$ define
\begin{align*}
  F_\tau (x)
  &:=
  \int_0^1 f^{(\tau)}(y) \frac{\tilB_\tau(x-y)}{\tau!} \rd y
  .
\end{align*}
Making use of the properties of the Bernoulli polynomials from Section~\ref{sec:bernoulli} we obtain the following results.
For $\tau=1,\ldots,\alpha-1$, using integration by parts on $F_\tau(x)$ gives
\begin{align*}
 \int_0^1 f^{(\tau)}(y) \frac{\tilB_\tau(x-y)}{\tau!} \rd y
 &=
 \left[f^{(\tau)}(y) \frac{-\tilB_{\tau+1}(x-y)}{(\tau+1)!}\right]_{y=0}^1
 -
 \int_0^1 f^{(\tau+1)}(y) \frac{-\tilB_{\tau+1}(x-y)}{(\tau+1)!} \rd y 
 \\
 \Leftrightarrow
 F_{\tau+1}(x) - F_\tau(x)
 &=
 \frac{\tilB_{\tau+1}(x)}{(\tau+1)!} \left( f^{(\tau)}(1)- f^{(\tau)}(0) \right)
 =
 \frac{\tilB_{\tau+1}(x)}{(\tau+1)!} \int_0^1 f^{(\tau+1)}(y) \rd{y}
 .
\end{align*}
We remark that the integration by parts,  
$\int u \rd v 
=
u v  - \int  v \rd u$,
is valid since $u = f^{(\tau)}(y)$ is absolutely continuous for all $\tau = 0, \ldots, \alpha-1$ and $\rmd v = (\tilB_\tau(x-y) / \tau!) \rd y$ is Lebesgue integrable \cite{R1976}.
Therefore, using a telescoping sum, we obtain
\begin{align}
  \notag
  F_\alpha(x) - F_1(x)
  =
  \sum_{\tau=1}^{\alpha-1} F_{\tau+1}(x) - F_\tau(x)
  &=
  \sum_{\tau=1}^{\alpha-1}
  \frac{\tilB_{\tau+1}(x)}{(\tau+1)!} \int_0^1 f^{(\tau+1)}(y) \rd{y}
  \\\label{eq:Falpha-F1}
  &=
  \sum_{\tau=2}^\alpha
  \frac{B_\tau(x)}{\tau!} \int_0^1 f^{(\tau)}(y) \rd{y}
  \quad
  \text{for $x \in [0,1]$}
  .
\end{align}
Note that the last line is obtained by shifting the index and since $\tilB_\tau(x) = B_\tau(x)$ for $\tau \ge 2$ and $x \in [0,1]$.
Furthermore, for $\tau=1$, we can also use integration by parts with the roles of $u$ and $\rmd v$ interchanged, carefully splitting the integral at the discontinuity of $\tilB_1$, such that for $x \in [0,1]$ we have
\begin{align*}
  F_1(x)
  &=
  \int_0^{x^-} \tilB_1(x-y) f'(y) \rd y 
  +
  \int_{x^+}^1 \tilB_1(x-y) f'(y) \rd y 
  \\
  &=
  \int_0^{x^-} B_1(x-y) f'(y) \rd y 
  +
  \int_{x^+}^1 B_1(x-y+1) f'(y) \rd y 
  \\
  &=
  \Big[B_1(x-y) f(y) \Big]_{y=0}^{x^-} + \int_0^{x^-} f(y) \rd y
  +
  \Big[B_1(x-y+1) f(y) \Big]_{y=x^+}^1 + \int_{x^+}^1 f(y) \rd y
  \\
  &=
  -f(x) + B_1(x)( f(1)- f(0)) + \int_0^1 f(y) \rd y,
\end{align*}
where $x^-$ and $x^+$ are the one-sided limits from left and right, respectively.
Substituting this into~\eqref{eq:Falpha-F1} then gives, for $x \in [0,1]$,
\begin{align}
  \notag
  F_\alpha(x)
  &=
  -f(x)
  +
  \sum_{\tau=1}^\alpha \frac{B_\tau(x)}{\tau!}
  \int_0^1 f^{(\tau)}(y) \rd y
  +
  \int_0^1 f(y) \rd y
  \\\label{eq:Falpha}
  &=
  -f(x)
  +
  \sum_{\tau=0}^\alpha \frac{B_\tau(x)}{\tau!}
  \int_0^1 f^{(\tau)}(y) \rd y
  \\\label{eq:Falpha'}
  &=
  \sum_{\substack{\setu \subseteq \{1\} \\ \overline{\setu} := \{1\} \setminus \setu}} (-1)^{|\overline{\setu}|}
  \sum_{\substack{\bstau_\setu \in \{0,\ldots,\alpha\}^{|\setu|} \\ \tau_j = 0 \text{ for } j \in \overline{\setu}}}
  \left[ \prod_{j \in \setu} \frac{B_{\tau_j}(x_j)}{\tau_j!} \right]
  \int_{[0,1]^{|\setu|}} f^{(\tau_1)}(x_1) \rd x_\setu
  ,
\end{align}
where the last line is already written in the notation of the multivariate statement since this is the form we want to use further down.
In other words we have obtained the following series expansion with remainder term, for $x \in [0,1]$,
\begin{align}\label{eq:Bseries-with-remainder}
  f(x)
  &=
  \sum_{\tau=0}^\alpha \frac{B_\tau(x)}{\tau!}
  \int_0^1 f^{(\tau)}(y) \rd y
  -
  \int_0^1 \frac{\tilB_\alpha(x-y)}{\alpha!} f^{(\alpha)}(y) \rd y
  .
\end{align}

To derive the multivariate claim we will apply~\eqref{eq:Falpha} recursively, starting from the following expression, where we apply~\eqref{eq:Falpha} for the first dimension in the form of~\eqref{eq:Falpha'} such that the sums over $\setu$ and $\bstau$ can be moved to the front.
The technique is the same as in Lemma~\ref{lem:BPM}, but for clarity we write it out in full:
\begin{align*}
  &\hspace*{-1mm}
  \int_{[0,1]^d} \left[ \prod_{j=1}^d \frac{\tilB_\alpha(x_j-y_j)}{\alpha!} \right]
    f^{(\alpha,\ldots,\alpha)}(\bsy) \rd\bsy
  \\
  &=
  \int_{[0,1]^{d-1}} \left[ \prod_{j=2}^d \frac{\tilB_\alpha(x_j-y_j)}{\alpha!} \right]
    \left[ \int_0^1 \frac{\tilB_\alpha(x_1-y_1)}{\alpha!} f^{(\alpha,\ldots,\alpha)}(\bsy) \rd y_1 \right] \rd\bsy_{-\{1\}}
  \\
  &=
  \int_{[0,1]^{d-1}} \left[ \prod_{j=2}^d \frac{\tilB_\alpha(x_j-y_j)}{\alpha!} \right]
    \Bigg[ -f^{(0,\alpha,\ldots,\alpha)}(x_1,\bsy_{-\{1\}})
  \\
  &\hspace*{6cm}
      + \sum_{\tau_1=0}^\alpha \frac{B_{\tau_1}(x_1)}{\tau_1!} \int_0^1 f^{(\tau_1,\alpha,\ldots,\alpha)}(y_1, \bsy_{-\{1\}}) \rd y_1 \Bigg] \rd\bsy_{-\{1\}}
  \\
  &=
  \sum_{\substack{\setu \subseteq \{1\} \\ \overline{\setu} := \{1\} \setminus \setu}}
  (-1)^{|\overline{\setu}|}
  \sum_{\substack{\bstau_\setu \in \{0,\ldots,\alpha\}^{|\setu|} \\ \tau_j = 0 \text{ for } j \in \overline{\setu}}}
  \left[ \prod_{j\in\setu} \frac{B_{\tau_j}(x_j)}{\tau_j!} \right]
  \\
  &\hspace*{3cm}
  \int_{[0,1]^{|\setu|}}
  \left[
  \int_{[0,1]^{d-1}} \left[ \prod_{j=2}^d \frac{\tilB_\alpha(x_j-y_j)}{\alpha!} \right]
    f^{(\tau_1,\alpha,\ldots,\alpha)}(x_1,\bsy_{-\{1\}})
  \rd\bsy_{-\{1\}}
  \right]
  \rd \bsx_\setu
  ,
\intertext{where $\bsy_{-\{1\}} = (y_2,\ldots,y_d)$ is the complement w.r.t.\ the full dimensional set, i.e., $\{1,\ldots,d\} \setminus \{1\}$; at this point the part under the outside integral on the second line of the previous expression is of the form we started from and we repeat the procedure now for the second dimension to obtain}
  \\
  &=
  \sum_{\substack{\setu \subseteq \{1\} \\ \overline{\setu} := \{1\} \setminus \setu}}
  (-1)^{|\overline{\setu}|}
  \sum_{\substack{\bstau_\setu \in \{0,\ldots,\alpha\}^{|\setu|} \\ \tau_j = 0 \text{ for } j \in \overline{\setu}}}
  \left[ \prod_{j\in\setu} \frac{B_{\tau_j}(x_j)}{\tau_j!} \right]
  \int_{[0,1]^{|\setu|}} \Bigg[
  \\
  &\hspace*{3mm}
  \int_{[0,1]^{d-2}}
  \left[ \prod_{j=3}^d \frac{\tilB_\alpha(x_j-y_j)}{\alpha!} \right]
  \left[ 
    \int_0^1 \frac{\tilB_\alpha(x_2-y_2)}{\alpha!}
    f^{(\tau_1,\alpha,\ldots,\alpha)}(x_1, y_2, \bsy_{-\{1,2\}})
    \rd y_2
  \right]
  \rd\bsy_{-\{1,2\}}
  \Bigg] \rd \bsx_\setu
  \\
  &=
  \sum_{\substack{\setu \subseteq \{1\} \\ \overline{\setu} := \{1\} \setminus \setu}}
  (-1)^{|\overline{\setu}|}
  \sum_{\substack{\bstau_\setu \in \{0,\ldots,\alpha\}^{|\setu|} \\ \tau_j = 0 \text{ for } j \in \overline{\setu}}}
  \left[ \prod_{j\in\setu} \frac{B_{\tau_j}(x_j)}{\tau_j!} \right]
  \sum_{\substack{\setv \subseteq \{2\} \\ \overline{\setv} := \{2\} \setminus \setv}}
  (-1)^{|\overline{\setv}|}
  \sum_{\substack{\bstau_{\setv} \in \{0,\ldots,\alpha\}^{|\setv|} \\ \tau_j = 0 \text{ for } j \in \overline{\setv}}}
  \left[ \prod_{j\in\setv} \frac{B_{\tau_j}(x_j)}{\tau_j!} \right]
  \\
  &\hspace*{3mm}
  \int_{[0,1]^{|\setu|}} \int_{[0,1]^{|\setv|}} \left[
  \int_{[0,1]^{d-2}}
  \left[ \prod_{j=3}^d \frac{\tilB_\alpha(x_j-y_j)}{\alpha!} \right] 
    f^{(\tau_1,\tau_2,\alpha,\ldots,\alpha)}(x_1, x_2, \bsy_{-\{1,2\}})
  \rd\bsy_{-\{1,2\}}
  \right] \rd \bsx_{\setv} \rd \bsx_\setu
  \\
  &=
  \sum_{\substack{\setw \subseteq \{1,2\} \\ \overline\setw := \{1,2\} \setminus \setw}}
  (-1)^{|\overline\setw|}
  \sum_{\substack{\bstau_\setw \in \{0,\ldots,\alpha\}^{|\setw|} \\ \tau_j = 0 \text{ for } j \in \overline\setw}}
  \left[ \prod_{j\in\setw} \frac{B_{\tau_j}(x_j)}{\tau_j!} \right]
  \\
  &\hspace*{1cm}
  \int_{[0,1]^{|\setw|}}
  \left[ \int_{[0,1]^{d-2}}
  \left[ \prod_{j=3}^d \frac{\tilB_\alpha(x_j-y_j)}{\alpha!} \right] 
    f^{(\tau_1,\tau_2,\alpha,\ldots,\alpha)}(x_1, x_2, \bsy_{-\{1,2\}})
  \rd\bsy_{-\{1,2\}}
  \right]
  \rd \bsx_\setw
  ,
\intertext{and therefore}
  \\
  &=
  \sum_{\setu \subseteq \{1,\ldots,d\}}
  (-1)^{d-|\setu|}
  \sum_{\substack{\bstau_\setu \in \{0,\ldots,\alpha\}^{|\setu|} \\ \tau_j = 0 \text{ for } j \notin \setu}}
    \left[ \prod_{j\in\setu} \frac{B_{\tau_j}(x_j)}{\tau_j!} \right]
    \int_{[0,1]^{|\setu|}} f^{(\bstau)}(\bsx) \rd \bsx_\setu
  .
\end{align*}
This concludes the proof.
\end{proof}

It is useful to compare the representation of a function~$f$ as in the previous statement to another representation from the literature in \cite{DKLNS2014}.

\begin{proposition}\label{prop:representations}
  For a function $f$ on the unit cube having mixed partial derivatives up to order $\alpha-1$ in each variable which are absolutely continuous and of which the derivatives with the highest order $\alpha$ are integrable, we have the following equivalent representations for all $\bsx \in [0,1]^d$:\\
  \ri series representation in terms of Bernoulli polynomials with remainder term
\begin{multline*}
  f(\bsx)
  =
  \sum_{\emptyset \ne \setu \subseteq \{1,\ldots,d\}}
  (-1)^{|\setu|+1}
  \sum_{\substack{\bstau_\setu \in \{0,\ldots,\alpha\}^{|\setu|} \\ \tau_j = 0 \text{ for } j \notin \setu}}
  \left[ \prod_{j\in\setu} \frac{B_{\tau_j}(x_j)}{\tau_j!} \right]
  \int_{[0,1]^{|\setu|}} f^{(\bstau)}(\bsx_{-\setu},\bsy_\setu) \rd \bsy_{\setu}
  \\
  + (-1)^d
  \int_{[0,1]^d}
  \left[ \prod_{j=1}^d \frac{\tilB_\alpha(x_j-y_j)}{\alpha!} \right]
  f^{(\alpha,\ldots,\alpha)}(\bsy) \rd \bsy
  ,
\end{multline*}
  \rii series representation in terms of Bernoulli polynomials with mixed remainder terms
\begin{multline*}
  f(\bsx)
  =
  \sum_{\setu \subseteq \{1,\ldots,d\}}
  (-1)^{d-|\setu|}
  \sum_{\substack{\bstau_\setu \in \{0,\ldots,\alpha\}^{|\setu|} \\ \tau_j = \alpha \text{ for } j \notin \setu}}
  \left[ \prod_{j\in\setu} \frac{B_{\tau_j}(x_j)}{\tau_j!} \right]
  \int_{[0,1]^d}
  \left[ \prod_{j\notin\setu} \frac{\tilB_\alpha(x_j-y_j)}{\alpha!} \right]
  f^{(\bstau)}(\bsy)
  \rd\bsy
  .
\end{multline*}
\end{proposition}
\begin{proof}
Representation \ri follows immediately from Lemma~\ref{lem:bernpartint}.
For representation~\rii we start from the one-dimensional form~\eqref{eq:Bseries-with-remainder}, which, says that for all such $f$ and for all $x \in [0,1]$ we have the representation
\begin{align*}
  f(x)
  &=
  \sum_{\tau=0}^\alpha \frac{B_\tau(x)}{\tau!}
  \int_0^1 f^{(\tau)}(y) \rd y
  -
  \int_0^1 \frac{\tilB_\alpha(x-y)}{\alpha!} f^{(\alpha)}(y) \rd y
  \\
  &=
  \sum_{\substack{\setu \subseteq \{1\} \\ \overline{\setu} := \{1\} \setminus \setu}}
  (-1)^{|\overline{\setu}|}
  \sum_{\substack{\bstau_\setu \{0,\ldots,\alpha\}^{|\setu|} \\ \tau_j = \alpha \text{ for } j \in \overline{\setu}}}
  \left[\prod_{j\in\setu} \frac{B_{\tau_j}(x_j)}{\tau_j!} \right]
  \int_0^1
  \left[ \prod_{j\in\overline{\setu}} \frac{\tilB_\alpha(x_j-y_j)}{\alpha!} \right]
  f^{(\tau_1)}(y_1)
  \rd{y_1}
  .
\end{align*}
If we apply this expansion dimension by dimension then we arrive at the claimed form with the proof technique of Lemma~\ref{lem:BPM}.
\end{proof}

We note that representation \rii from Proposition~\ref{prop:representations} is equivalent to the following form which was used in \cite[Equation~$(3.13)$ in Theorem~3.5]{DKLNS2014}
\begin{multline*}
  f(\bsx)
  =
  \sum_{\setw \subseteq \{1,\ldots,d\}}
  \;\; \sum_{\setv \subseteq \setw} \;\;
  (-1)^{|\setv|}
  \sum_{\substack{\bstau_{\setw \setminus \setv} \in \{1,\ldots,\alpha\}^{|\setw \setminus \setv|} \\ \tau_j = \alpha \text{ for } j \in \setv \\ \tau_j = 0 \text{ for } j \notin \setw}}
  \left[ \prod_{j \in \setw \setminus \setv} \frac{B_{\tau_j}(x_j)}{\tau_j!} \right]
   \\
   \int_{[0,1]^d}
   \left[ \prod_{j \in \setv} \frac{\tilB_\alpha(x_j-y_j)}{\alpha!} \right]
   f^{(\bstau)}(\bsy)
   \rd \bsy
   ,
\end{multline*}
by considering $j\notin\setu$ in the formula for \rii to be $j \in \setv$ in this formula, and using $B_0(x)/0! = 1$ and we already used $\tilB_\alpha(x-y) = (-1)^\alpha \tilB_\alpha(y-x)$ in comparison with the original formula in~\cite{DKLNS2014}.
The latter formula is advantageous if considering weighted spaces or when access to the ANOVA decomposition of~$f$ is required.
The advantage of representation \ri over \rii and the equivalent form just discussed is that \ri does not contain products with other Bernoulli polynomials when considering the term with $\prod_{j=1}^d \tilB_\alpha(x_j-y_j) / \alpha!$.
We will make use of this form (modified for arbitrary boxes as in the next subsection) in Section~\ref{sec:error-analysis} to bound the projection error.

Now we are ready to state the important result of this section, namely, that the Bernoulli polynomial method is an orthogonal projection.

\begin{proposition}\label{prop:projection}
 Let $f\in\calH(\KSob_{\alpha,d})$, then we have
 \[
 F(\bsy)
 =
 \left\langle f , \KKor_{\alpha,d}(\cdot,\bsy) \right\rangle_{\KSob_{\alpha,d}}
 ,
 \]
 where $F$ is the function obtained by applying the Bernoulli polynomial method~\eqref{eq:BPM} to~$f$.
 Therefore, the transformed function $F$ does not have a bigger norm than the original function $f$,
 \[
  \|F\|_{\KKor_{\alpha,d}} 
  =
  \|F\|_{\KSob_{\alpha,d}} 
  \le
  \|f\|_{\KSob_{\alpha,d}}
  .
 \]
\end{proposition}
\begin{proof}
For this proof, we introduce a shorthand notation for the result of Lemma~\ref{lem:bernpartint}.
Instead of writing $\bstau_\setu \in \{0,\ldots,\alpha\}^{|\setu|}$ we write the multiindex $\bstau \in \{0,\ldots,\alpha,\overline{0}\}^d$ where now whenever $\tau_j = \overline{0}$ this would be $j \notin \setu$ in the original notation.
Lemma~\ref{lem:bernpartint} can then be written as
\begin{multline*}
  \sum_{\substack{\bstau \in \{0,\ldots,\alpha,\overline{0}\}^{d} \\ \overline{\setu} := \{j: \tau_j = \overline{0} \}}}
  (-1)^{|\overline{\setu}|}
  \left[ \prod_{j \in -\overline{\setu}} \frac{B_{\tau_j}(x_j)}{\tau_j!} \right]
  \int_{[0,1]^{d-|\overline{\setu}|}}
  f^{(\bstau)}(\bsx) \rd \bsx_{-\overline{\setu}}
  \\=
  \int_{[0,1]^d}
  \left[ \prod_{j=1}^d \frac{\tilB_\alpha(x_j-y_j)}{\alpha!} \right]
  f^{(\alpha,\ldots,\alpha)}(\bsy) \rd \bsy
  ,
\end{multline*}
where we interpret $\overline{0} = 0$ for the derivatives and where $-\overline{\setu} = \{1,\ldots,d\} \setminus \overline{\setu}$.
In the following we will use this with $\bsx$ and $\bsy$ interchanged and where the derivatives of order $\alpha$ in the right hand side appear for a certain subset $\setw \subseteq \{1,\ldots,d\}$ such that the left hand side is modified accordingly.
More precisely, for given $\setw \subseteq \{1,\ldots,d\}$ and $\bstau_{-\setw} = \bstau_{\{1,\ldots,d\}\setminus\setw} = \bszero$ we will use
\begin{multline*}
  \sum_{\substack{\bstau_\setw \in \{0,\ldots,\alpha,\overline{0}\}^{|\setw|} \\ \overline{\setu} := \{j \in \setw: \tau_j = \overline{0} \} \\ \tau_j = 0 \text{ for } j \notin \setw}}
  (-1)^{|\overline{\setu}|}
  \left[ \prod_{j\in\setw\setminus\overline{\setu}} \frac{B_{\tau_j}(y_j)}{\tau_j!} \right]
  \int_{[0,1]^{|\setw|-|\overline{\setu}|}}
  f^{(\bstau)}(\bsy_{\setw\setminus\overline{\setu}}, \bsy_{\overline{\setu}}, \bsx_{-\setw}) \rd \bsy_{\setw\setminus\overline{\setu}}
  \\=
  \int_{[0,1]^{|\setw|}}
  \left[ \prod_{j\in\setw} \frac{\tilB_\alpha(y_j-x_j)}{\alpha!} \right]
  f^{(\bsalpha_\setw,\bszero_{-\setw})}(\bsx_\setw,\bsx_{-\setw}) \rd \bsx_\setw
  ,
\end{multline*}
which, after integrating on both sides over $\bsx_{-\setw}$ gives
\begin{multline}\label{eq:bern1-w}
  \sum_{\substack{\bstau_\setw \in \{0,\ldots,\alpha,\overline{0}\}^{|\setw|} \\ \overline{\setu} := \{j \in \setw: \tau_j = \overline{0} \} \\ \tau_j = 0 \text{ for } j \notin \setw}}
  (-1)^{|\overline{\setu}|}
  \left[ \prod_{j\in\setw\setminus\overline{\setu}} \frac{B_{\tau_j}(y_j)}{\tau_j!} \right]
  \int_{[0,1]^{d-|\overline{\setu}|}}
  f^{(\bstau)}(\bsy) \rd \bsy_{-\overline{\setu}}
  \\=
  \int_{[0,1]^d}
  \left[ \prod_{j\in\setw} \frac{\tilB_\alpha(y_j-x_j)}{\alpha!} \right]
  f^{(\bsalpha_\setw,\bszero_{-\setw})}(\bsx) \rd \bsx
  .
\end{multline}

We now start by expanding the inner product of the Sobolev space $\KSob_{\alpha,d}$ for $f$ and the kernel of the Korobov space $\KKor_{\alpha,d}$ with one leg fixed to $\bsy$:
 \begin{align*}
    &\left\langle f , \KKor_{\alpha,d}(\cdot,\bsy) \right\rangle_{\KSob_{\alpha,d}}
    \\
    &\quad=
    \sum_{ \substack{ \bstau \in \{0,\ldots,\alpha \}^d \\ \setw := \{j: \tau_j=\alpha \}} }  
    \int_{[0,1]^{|\setw|}} 
    \left(
    \int_{[0,1]^{d-|\setw|}} f^{(\bstau)} (\bsx) \rd \bsx_{-\setw}
    \right) 
    \left(
    \int_{[0,1]^{d-|\setw|}} {\KKor_{\alpha,d}}^{(\bstau)}  (\bsx,\bsy) \rd \bsx_{-\setw}
    \right) 
    \rd \bsx_\setw,    
\end{align*}
here we have, cf.~\eqref{eq:KKor} and~\eqref{eq:Btilde-integral},
\[
 {\KKor_{\alpha,d}}^{(\bstau)}(\bsx,\bsy)
 =
 \prod_{j=1}^d \left( \delta_{\tau_j,0} \, [B_0]^2 + (-1)^{\alpha+1} \frac{\tilB_{2\alpha-\tau_j}(x_j-y_j)}{(2\alpha-\tau_j)!}  \right)
 ,
\]
where we have put $B_0 = B_0(x) = 1$ for ease of generalization to the box later where $B_0^{[a_j,b_j]}(x) = 1/(b_j-a_j)$ (and thus integrates to one on $[a_j,b_j]$).
Hence integrating with respect to $x_j$ with $j \in -\setw$, i.e., $\tau_j \ne \alpha$, vanishes unless $\tau_j = 0$.
Therefore only $\tau_j = 0$ and $\tau_j = \alpha$ remain and we obtain, with $-\setw = \{1,\ldots,d\} \setminus \setw$ and using~\eqref{eq:bern1-w},
 \begin{align}
    \notag
    &\left\langle f, \KKor_{\alpha,d}(\cdot,\bsy) \right\rangle_{\KSob_{\alpha,d}}
    \\\notag
    &\quad=
    \sum_{ \substack{ \bstau \in \{0,\alpha \}^d \\ \setw := \{j: \tau_j=\alpha \}} }
    \int_{[0,1]^{d}} 
    f^{(\bstau)} (\bsx) 
    \left[ \prod_{j\in-\setw} B_0(y_j) \right]
    \left[ \prod_{j\in \setw} (-1)^{\alpha+1} \frac{\tilB_\alpha (x_j-y_j)}{\alpha!} \right]
    \rd \bsx 
    \\\notag
    &\quad=
    \sum_{\setw \subseteq \{1,\ldots,d\}}
    (-1)^{|\setw|}
    \left[ \prod_{j\in-\setw} \frac{B_0(y_j)}{0!} \right]
    \int_{[0,1]^{d}} 
    \left[ \prod_{j\in \setw} \frac{\tilB_\alpha (y_j-x_j)}{\alpha!} \right]
    f^{(\bsalpha_\setw,\bszero_{-\setw})}(\bsx)
    \rd \bsx 
    \\\notag
    &\quad=
    \sum_{\setw \subseteq \{1,\ldots,d\}}
    (-1)^{|\setw|}
    \left[ \prod_{j\in-\setw} \frac{B_0(y_j)}{0!} \right]
    \sum_{\substack{\bstau_\setw \in \{ 0,\ldots,\alpha,\overline{0}\}^{|\setw|} \\ \overline{\setu} := \{j\in\setw: \tau_j=\overline{0}\} \\ \tau_j = 0 \text{ for } j \notin \setw} }
  \hspace{-3mm}  (-1)^{|\overline{\setu}|} 
    \left[ \prod_{j \in \setw\setminus\overline{\setu}} \frac{B_{\tau_j}(y_j)}{\tau_j!} \right]
      \int_{[0,1]^{d-|\overline{\setu}|}}
     \hspace{-3mm} f^{(\bstau)}(\bsy) 
      \rd \bsy_{-\overline{\setu}}
    \\\notag
    &\quad=
    \sum_{\setw \subseteq \{1,\ldots,d\}}
    (-1)^{|\setw|}
    \sum_{\substack{\bstau_\setw \in \{ 0,\ldots,\alpha,\overline{0}\}^{|\setw|} \\ \overline{\setu} := \{j\in\setw: \tau_j=\overline{0}\} \\ \tau_j = 0 \text{ for } j \notin \setw} }
    (-1)^{|\overline{\setu}|}
    \left[ \prod_{j \in -\overline{\setu}} \frac{B_{\tau_j}(y_j)}{\tau_j!} \right]
      \int_{[0,1]^{d-|\overline{\setu}|}}
      f^{(\bstau)}(\bsy)
      \rd \bsy_{-\overline{\setu}}
     \\\label{eq:BPM-overlinezero}
    &\quad=
    \sum_{\substack{\bstau \in \{1,\ldots,\alpha,\overline{0}\}^d \\ \overline{\setu} := \{j: \tau_j = \overline{0} \}}}
    (-1)^{d-|\overline{\setu}|}
    \left[ \prod_{j \in -\overline{\setu}} \frac{B_{\tau_j}(y_j)}{\tau_j!} \right]
    \int_{[0,1]^{d-|\overline{\setu}|}} f^{(\bstau)}(\bsy) \rd\bsy_{-{\overline{\setu}}}
    ,
 \end{align}
 where for the last equality we used the inclusion-exclusion principle
 \begin{align*}
 \sum_{\bstau \in \{1,\ldots,\alpha,\overline{0} \} ^d } A(\bstau)
 &=
 \sum_{\bstau \in \{0,\ldots,\alpha,\overline{0} \} ^d } A(\bstau) - \sum_{j=1}^d  \sum_{\substack{\bstau_{-\{j\}} \in \{0,\ldots,\alpha,\overline{0} \} ^{d-1} \\ \tau_j=0}} A(\bstau)+
 \cdots
 +
 (-1)^d \sum_{\bstau= \bszero} A(\bstau)
 \\
 &=
 \sum_{\setw \subseteq \{1,\ldots,d\}}
  (-1)^{d-|\setw|}
  \sum_{\substack{\bstau_\setw \in \{0,\ldots,\alpha,\overline{0}\}^{|\setw|} \\ \tau_j = 0 \text{ for } j \notin \setw}}
  A(\bstau)
 ,
 \end{align*}
 for any function $A(\bstau)$ defined for multiindices $\bstau \in \{0,\ldots,\alpha,\overline{0} \}^d$, where here
 \begin{align*}
   A(\bstau)
   &=
    (-1)^{d-|\overline{\setu}(\bstau)|}
    \left[ \prod_{j \in -\overline{\setu}(\bstau)} \frac{B_{\tau_j}(y_j)}{\tau_j!} \right]
      \int_{[0,1]^{d-|\overline{\setu}(\bstau)|}}
      f^{(\bstau)}(\bsy)
      \rd \bsy_{-\overline{\setu}(\bstau)}
   \quad \text{with} \quad
   \overline{\setu}(\bstau)
   :=
   \{ j : \tau_j = \overline{0} \}
 \end{align*}
and we used $(-1)^{|\setw|} (-1)^{|\overline{\setu}|} = (-1)^{2d} (-1)^{-|\setw|} (-1)^{-|\overline{\setu}|} = (-1)^{d-|\setw|} (-1)^{d-|\overline{\setu}}|$.
Note that in the derivation we wrote $\overline{\setu} := \{ j \in \setw : \tau_j = \overline{0}\}$ under the sums with $j \in \setw$ for clarity, but here in fact we could have just written $ \{ j : \tau_j = \overline{0}\}$ in the derivation as well.
We conclude the proof by noting that~\eqref{eq:BPM-overlinezero} and~\eqref{eq:BPM} are the same by our definition of the symbol~$\overline{0}$.
\end{proof}
\begin{figure}
\centering
\includegraphics[scale=1]{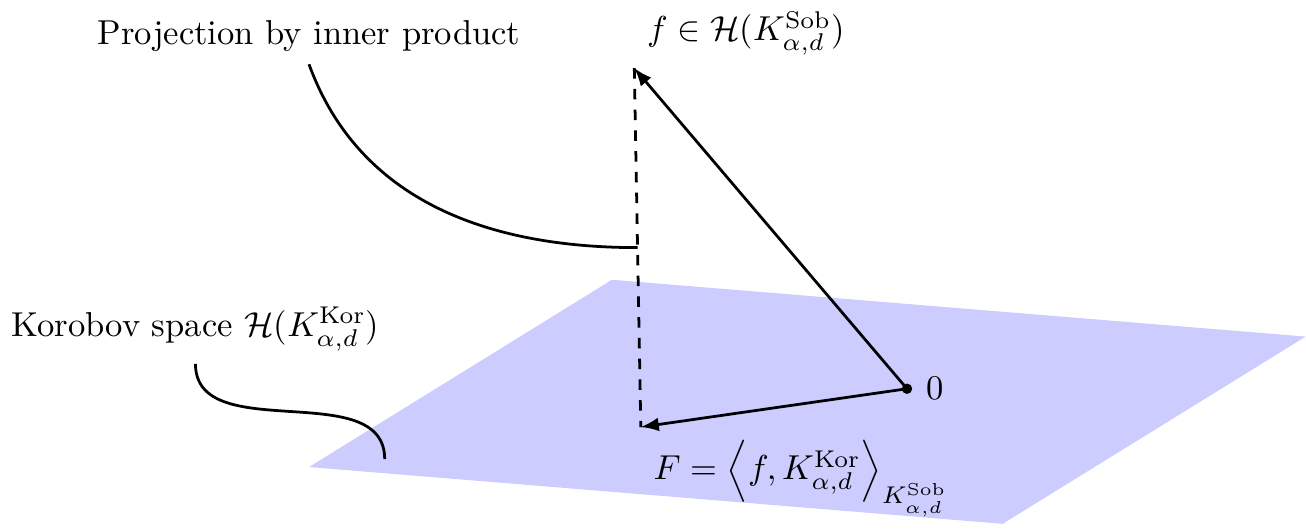}
\caption{The Bernoulli polynomial method as an orthogonal projection.}
\end{figure}

\subsection{The Bernoulli polynomial method on a box}

 We first generalize Lemma~\ref{lem:bernpartint} to a general box $[\bsa,\bsb]$. 

 \begin{lemma}\label{lem:bernpartintab}
   For a function $f$ on the box $[\bsa,\bsb]$ having mixed partial derivatives up to order $\alpha-1$ in each variable which are absolutely continuous and of which the derivatives with the highest order $\alpha$ are integrable, the following holds for all $\bsx \in [\bsa,\bsb]$:
   \begin{multline*}
  \sum_{\setu \subseteq \{1,\ldots,d\}}
  (-1)^{d-|\setu|}
  \sum_{\substack{\bstau_\setu \in \{0,\ldots,\alpha\}^{|\setu|} \\ \tau_j = 0 \text{ for } j \notin \setu}}
  \left[ \prod_{j\in\setu} \frac{B^{[a_j,b_j]}_{\tau_j}(x_j)}{\tau_j!} \right]
  \int_{[\bsa_{\setu},\bsb_{\setu}]} f^{(\bstau)}(\bsx) \rd \bsx_{\setu}
  \\
  =
  \int_{[\bsa,\bsb]}
  \left[ \prod_{j=1}^d \frac{\tilB^{[a_j,b_j]}_\alpha(x_j-y_j+a_j)}{\alpha!} \right]
  f^{(\alpha,\ldots,\alpha)}(\bsy) \rd \bsy
  .
\end{multline*}
 \end{lemma}
 \begin{proof}
 The proof is similar to the proof of Lemma~\ref{lem:bernpartint}.
 We show the one-dimensional results which follow from the properties of the scaled Bernoulli polynomials from Section~\ref{sec:bernoulli}.
 Define
 \begin{align*}
   F^{[a,b]}_\tau(x)
   &:=
   \int_a^b f^{(\tau)}(y) \frac{\tilB_\tau^{[a,b]}(x-y+a)}{\tau!} \rd y
   ,
 \end{align*}
 then, as is the case for the interval $[0,1]$, we here obtain, for $\tau=1,\ldots,\alpha-1$, using integration by parts,
 \begin{align*}
   F^{[a,b]}_{\tau+1}(x) - F^{[a,b]}_\tau(x)
   &=
   \frac{\tilB^{[a,b]}_{\tau+1}(x)}{(\tau+1)!} \int_a^b f^{(\tau+1)}(y) \rd y
   .
 \end{align*}
 We also obtain for $F^{[a,b]}_1$ using integration by parts with the terms interchanged and for $x\in[a,b]$
 \begin{align*}
  \int_a^b \tilB^{[a,b]}_1(x-y+a) f'(y) \rd y
  &=
  \int_a^{x^-} \tilB^{[a,b]}_1(x-y+a) f'(y) \rd y 
  +
  \int_{x^+}^b \tilB^{[a,b]}_1(x-y+a) f'(y) \rd y 
  \\
  &=
  \int_a^{x^-} B^{[a,b]}_1(x-y+a) f'(y) \rd y 
  +
  \int_{x^+}^b B^{[a,b]}_1(x-y+b) f'(y) \rd y 
  \\
  &=
  \left[B^{[a,b]}_1(x-y+a) f(y) \right]_{y=a}^{x^-}
  +
  \int_a^{x^-} \frac{f(y)}{b-a} \rd y
  \\
  &\quad+
  \left[B^{[a,b]}_1(x-y+b) f(y) \right]_{y=x^+}^{b}
  +
  \int_{x^+}^b \frac{f(y)}{b-a} \rd y
  \\
  &=
  -f(x) + B^{[a,b]}_1(x) (f(b)-f(a)) + B^{[a,b]}_0(x) \int_a^b f(y) \rd y
  ,
\end{align*}
where we used the scaling of our Bernoulli polynomials on the box, cf.~\eqref{eq:def-Bab}.
 Therefore we have the following representation for one-dimensional functions on $[a,b]$
 \begin{align}\label{eq:Bseries-with-remainder-ab}
   f(x)
   &=
   \sum_{\tau=0}^\alpha \frac{B^{[a,b]}_\tau(x)}{\tau!} \int_a^b f^{(\tau)}(y) \rd y
   -
   \int_a^b \frac{\tilB^{[a,b]}_\alpha(x-y+a)}{\alpha!} f^{(\alpha)}(y) \rd y
   .
 \end{align}
 Using the same induction argument as in the proof of Lemma~\ref{lem:bernpartint} the result for the multivariate case now follows.
 \end{proof}

Using this result we can represent functions in $\calH(\KSobab_{\alpha,d})$ similarly as those in $\calH(\KSob_{\alpha,d})$, cf.\ Proposition~\ref{prop:representations}.

\begin{proposition}\label{prop:representations-ab}
  For a function $f$ on the box $[\bsa,\bsb]$ having mixed partial derivatives up to order $\alpha-1$ in each variable which are absolutely continuous and of which the derivatives with the highest order $\alpha$ are integrable, we have the following equivalent representations for all $\bsx \in [\bsa,\bsb]$: \\
  \ri series representation in terms of Bernoulli polynomials with remainder term
\begin{multline*}
  f(\bsx)
  =
  \sum_{\emptyset \ne \setu \subseteq \{1,\ldots,d\}}
  (-1)^{|\setu|+1}
  \sum_{\substack{\bstau_\setu \in \{0,\ldots,\alpha\}^{|\setu|} \\ \tau_j = 0 \text{ for } j \notin \setu}}
  \left[ \prod_{j\in\setu} \frac{B^{[a_j,b_j]}_{\tau_j}(x_j)}{\tau_j!} \right]
  \int_{[\bsa_\setu,\bsb_\setu]} f^{(\bstau)}(\bsx_{-\setu},\bsy_\setu) \rd \bsy_{\setu}
  \\
  + (-1)^d
  \int_{[\bsa,\bsb]}
  \left[ \prod_{j=1}^d \frac{\tilB^{[a_j,b_j]}_\alpha(x_j-y_j+a_j)}{\alpha!} \right]
  f^{(\alpha,\ldots,\alpha)}(\bsy) \rd \bsy
  ,
\end{multline*}
  \rii series representation in terms of Bernoulli polynomials with mixed remainder terms
\begin{multline*}
  f(\bsx)
  =
  \sum_{\setu \subseteq \{1,\ldots,d\}}
  (-1)^{d-|\setu|}
  \sum_{\substack{\bstau_\setu \in \{0,\ldots,\alpha\}^{|\setu|} \\ \tau_j = \alpha \text{ for } j \notin \setu}}
  \\
  \left[ \prod_{j\in\setu} \frac{B^{[a_j,b_j]}_{\tau_j}(x_j)}{\tau_j!} \right]
  \int_{[\bsa,\bsb]}
  \left[ \prod_{j\notin\setu} \frac{\tilB^{[a_j,b_j]}_\alpha(x_j-y_j+a_j)}{\alpha!} \right]
  f^{(\bstau)}(\bsy)
  \rd\bsy
  .
\end{multline*}
\end{proposition}
\begin{proof}
  The proof is similar to the proof of Proposition~\ref{prop:representations}, making use of the one-dimensional representation~\eqref{eq:Bseries-with-remainder-ab}.
\end{proof}

We are now finally in a position to show that $\KSobab_{\alpha,d}$ is in fact the reproducing kernel of $\calH(\KSobab_{\alpha,d})$ by making use of the representations just given.

\begin{proposition}\label{prop:kernelSobab}
  The function $\KSobab_{\alpha,d}$, given in~\eqref{eq:Sob-ab-kernel}, is the reproducing kernel corresponding to the inner product~\eqref{eq:Sob-ab-innerproduct} of the space $\calH(\KSobab_{\alpha,d})$.
  That is
  \begin{align*}
    \langle f , \KSobab_{\alpha,d}(\cdot,\bsy) \rangle_{\KSobab_{\alpha,d}}
    =
    f(\bsy)
    \qquad
    \forall \bsy \in [\bsa,\bsb]
    \text{ and }
    \forall f \in \calH(\KSobab_{\alpha,d})
    .
  \end{align*}
 \end{proposition}
 \begin{proof}
   We need to calculate 
   \begin{multline*}
    \langle f , \KSobab_{\alpha,d}(\cdot,\bsy) \rangle_{\KSobab_{\alpha,d}} 
    \\
    =
      \hspace{-2.5mm} \sum_{ \substack{ \bstau \in \{0,\ldots,\alpha \}^d \\ \setw := \{j: \tau_j=\alpha \}} } 
    \int_{[\bsa_\setw,\bsb_\setw]} 
    \left(
    \int_{[\bsa_{-\setw},\bsb_{-\setw}]} f^{(\bstau)} (\bsx) \rd \bsx_{-\setw}
    \right) 
    \left(
    \int_{[\bsa_{-\setw},\bsb_{-\setw}]} {\KSobab_{\alpha,d}}^{(\bstau)}  (\bsx,\bsy) \rd \bsx_{-\setw}
    \right) 
    \rd \bsx_\setw
    .
   \end{multline*}
   For $\bstau \in \{0,\ldots,\alpha\}^d$, we have
   \begin{multline*}
     {\KSobab_{\alpha,d}}^{(\bstau)}(\bsx, \bsy)
     =
     \prod_{j=1}^d
     \Bigg(
     \sum_{\tau'=\tau_j}^{\alpha-1} \frac{B_{\tau'-\tau_j} ^{[a_j,b_j]} (x_j) }{(\tau'-\tau_j)!} \frac{B_{\tau'} ^{[a_j,b_j]} (y_j)}{\tau'!} 
     + 
     (b_j-a_j) \frac{B_{\alpha-\tau_j} ^{[a_j,b_j]} (x_j)}{(\alpha-\tau_j)!} \frac{B_\alpha^{[a_j,b_j]}(y_j)}{\alpha!}
     \\+
     (-1)^{\alpha+1} \frac{\tilB_{2\alpha-\tau_j}^{[a_j,b_j]}(x_j-y_j+a_j)}{(2\alpha-\tau_j)!}
     \Bigg)
     .
   \end{multline*}
Therefore, integrating over the $x_j$ for which $\tau_j\ne\alpha$ gives
\begin{multline*}
  \int_{[\bsa_{-\setw},\bsb_{-\setw}]} {\KSobab_{\alpha,d}}^{(\bstau)}(\bsx,\bsy)\rd \bsx_{-\setw}
  \\=
  \left[\prod_{j\notin\setw} \frac{B_{\tau_j} ^{[a_j,b_j]} (y_j)}{\tau_j!} \right] \prod_{j\in\setw} \left( \frac{B_\alpha ^{[a_j,b_j]} (y_j)}{\alpha!}   + (-1)^{\alpha+1} \frac{\tilB_\alpha^{[a_j,b_j]}(x_j-y_j+a_j)}{\alpha!}   \right)
  .
\end{multline*}
It follows that
   \begin{multline*}
    \langle f , \KSobab_{\alpha,d}(\cdot,\bsy) \rangle_{\KSobab_{\alpha,d}} 
    =
    \sum_{ \substack{ \bstau \in \{0,\ldots,\alpha \}^d \\ \setw := \{j: \tau_j=\alpha \}} }
    \int_{[\bsa_\setw,\bsb_\setw]} 
    \left(
    \int_{[\bsa_{-\setw},\bsb_{-\setw}]} f^{(\bstau)} (\bsx) \rd \bsx_{-\setw}
    \right)
    \\
    \left[
    \prod_{j\notin\setw} \frac{B_{\tau_j} ^{[a_j,b_j]} (y_j)}{\tau_j!}
    \right]
    \prod_{j\in\setw} \left( \frac{B_\alpha ^{[a_j,b_j]} (y_j)}{\alpha!}   + (-1)^{\alpha+1} \frac{\tilB_\alpha^{[a_j,b_j]}(x_j-y_j+a_j)}{\alpha!}   \right)
    \rd \bsx_\setw 
    \end{multline*}
    which is equivalent to representation \rii in Proposition~\ref{prop:representations-ab}.
 \end{proof}

We now define the Bernoulli polynomial method on a box $[\bsa,\bsb]$.
For a function $f$ on the box $[\bsa,\bsb]$ having mixed partial derivatives up to order $\alpha-1$ in each variable which are absolutely continuous and of which the derivatives with the highest order $\alpha$ are integrable, we define
\begin{align}
    \notag
    &F^{[\bsa,\bsb]}(\bsx)
    \\
    \label{eq:BPM-ab}
    &\quad:=
    f(\bsx)
    +
    \sum_{\substack{\bszero \ne \bstau \in \{0,\ldots,\alpha\}^d \\ \setu := \supp(\bstau)}}
      (-1)^{|\setu|}
      \left[ \prod_{j \in \setu} \frac{B^{[a_j,b_j]}_{\tau_j}(x_j)}{\tau_j!} \right]
        \sum_{\setw\subseteq\setu} (-1)^{|\setw|}
          f^{(\bstau_\setu-\bsone_\setu, \bszero_{-\setu})}(\bsa_\setw,\bsb_{\setu\setminus\setw}, \bsx_{-\setu})
    \\
    \notag
    &\quad\hphantom{:}=
    f(\bsx)
    +
      \sum_{\substack{\bszero \ne \bstau \in \{0,\ldots,\alpha\}^d \\ \setu := \supp(\bstau)}}
      (-1)^{|\setu|}
      \left[ \prod_{j \in \setu} \frac{B^{[a_j,b_j]}_{\tau_j}(x_j)}{\tau_j!} \right]
      \int_{[\bsa_\setu,\bsb_\setu]} f^{(\bstau)}(\bsx) \rd\bsx_{\setu}
  .
\end{align}
The next result is the orthogonal projection property for the Bernoulli polynomial method on a box, cf.\ Proposition~\ref{prop:projection}.

\begin{proposition}\label{prop:projection-ab}
  Let $f \in \calH(\KSobab_{\alpha,d})$, then the function $F^{[\bsa,\bsb]}\in \calH(\KKorab_{\alpha,d})$, obtained by the Bernoulli polynomial method on the box $[\bsa,\bsb]$, cf.~\eqref{eq:BPM-ab}, is equivalent to the orthogonal projection on $\calH(\KKorab_{\alpha,d})$, i.e., for all $\bsx \in [\bsa,\bsb]$,
  \begin{align*}
    F^{[\bsa,\bsb]}(\bsx)
    &=
    \langle f, \KKorab_{\alpha,d}(\cdot,\bsx) \rangle_{\KSobab_{\alpha,d}}
    .
  \end{align*}
  Furthermore
  \begin{align*}
    \int_{[\bsa,\bsb]} F^{[\bsa,\bsb]}(\bsx) \rd\bsx
    &=
    \int_{[\bsa,\bsb]} f(\bsx) \rd\bsx
  \end{align*}
  and
  \begin{align*}
    \|F^{[\bsa,\bsb]}\|_{\KKorab_{\alpha,d}}
    =
    \|F^{[\bsa,\bsb]}\|_{\KSobab_{\alpha,d}}
    \le
    \|f\|_{\KSobab_{\alpha,d}}
    .
  \end{align*}
\end{proposition}
\begin{proof}
  The proof is the equivalent of the proof of Proposition~\ref{prop:projection} but then making use of Lemma~\ref{lem:bernpartintab}.
\end{proof}

\section{Error analysis for integration over $\R^{\lowercase{d}}$}\label{sec:error-analysis}

In this section, the integration error for approximating an integral over $\R^d$ using a scaled lattice rule is shown.
Our proposed algorithm is simple: first we truncate the domain to a box $[\bsa,\bsb]$,
 and then we apply a lattice rule to compute the integral $\int_{[\bsa,\bsb]} f(\bsx) \rd \bsx$.
 We assume that the function $f$ and its derivatives decay fast enough, so that we can exploit the fact that the bigger box we truncate to, the smaller the orthogonal complement of the projected function becomes.
We start with the following lemma.

  \begin{lemma}\label{lem:cuberror-ab}
   Let $\Lambda(\bsz,n)$ be the nodes of an $n$-point lattice rule with generating vector~$\bsz$.
   The integration error of using the scaled lattice rule, with nodes $\bsp^{[\bsa,\bsb]}_i$, to integrate a function $f\in \calH(\KSobab_{\alpha,d})$ is bounded by
 \begin{multline*}
    \left|
    \frac{\prod_{j=1}^d (b_j-a_j)}{n} \sum_{i=1}^n f(\bsp_i^{[\bsa,\bsb]})
    -
    \int_{[\bsa,\bsb]} f(\bsx) \rd \bsx
    \right|
    \\
    \le
    \|f\|_{\KSobab_{\alpha,d}}
    \;
    \left( \prod_{j=1}^d \max(1, b_j-a_j)^{\alpha+1/2} \right)
    \left( \sum_{\bszero \ne \bsh \in \Lambda^\perp(\bsz,n)} [r_{\alpha,d}(\bsh)]^{-2} \right)^{1/2} 
    \\+
   \left| \frac{\prod_{j=1}^d (b_j-a_j)}{n} \sum_{i=1}^n
   \left(
   F^{[\bsa,\bsb]}(\bsp_i^{[\bsa,\bsb]})-f(\bsp_i^{[\bsa,\bsb]})
   \right)
   \right|
   ,   
 \end{multline*}
 where the term $(\sum_{\bszero \ne \bsh \in \Lambda^\perp(\bsz,n)} [r_{\alpha,d}(\bsh)]^{-2})^{1/2}$ is the worst-case error for a standard lattice rule on the unit cube $[0,1]^d$.
 There exist lattice rules $\Lambda(\bsz,n)$ for any $n \ge 2$, such that the worst-case error for integration in $\KKor_{\alpha,d}$ on the unit cube $[0,1]^d$ is bounded like
 \[
  \left(\sum_{\bszero \ne \bsh \in \Lambda^\perp(\bsz,n)} [r_{\alpha,d}(\bsh)]^{-2} \right)^{1/2}
  \le
  C_{\alpha,d} \, \frac{(\log n)^{\alpha d}}{n^\alpha},
 \]
 where $C_{\alpha,d}$ is a constant only depending on $\alpha$ and~$d$.
\end{lemma}
\begin{proof}
  We use Proposition~\ref{prop:projection-ab} to obtain
  \begin{multline*}
    \frac{\prod_{j=1}^d (b_j-a_j)}{n} \sum_{i=1}^n f(\bsp_i^{[\bsa,\bsb]})
    -
    \int_{[\bsa,\bsb]} f(\bsx) \rd \bsx
    \\
    =
    \frac{\prod_{j=1}^d (b_j-a_j)}{n} \sum_{i=1}^n F^{[\bsa,\bsb]}(\bsp_i^{[\bsa,\bsb]})
    -
    \int_{[\bsa,\bsb]} F^{[\bsa,\bsb]}(\bsx) \rd \bsx
    \\
    -
    \frac{\prod_{j=1}^d (b_j-a_j)}{n} \sum_{i=1}^n 
    \left(
      F^{[\bsa,\bsb]}(\bsp_i^{[\bsa,\bsb]}) - f(\bsp_i^{[\bsa,\bsb]})
    \right)
    .
  \end{multline*}
  The result now follows from applying the triangle inequality and Proposition~\ref{prop:lattice-rule-error-ab} and the fact that $F^{[\bsa,\bsb]}$ is obtained by a projection such that $\|F^{[\bsa,\bsb]}\|_{\KKorab_{\alpha,d}} = \|F^{[\bsa,\bsb]}\|_{\KSobab_{\alpha,d}} \le \|f\|_{\KSobab_{\alpha,d}}$.
 For the existence of such lattice rules, we refer to~\cite{SJ1994}.
\end{proof}

The previous lemma tells us that the integration error of a non-periodic function on the box by a lattice rule is bounded by a sum of two terms: the integration error of the projected (and periodic) function $F^{[\bsa,\bsb]}$ and the projection error, which is the difference $F^{[\bsa,\bsb]} - f$.
We will first show that this projection error can be driven to zero by enlarging the box $[\bsa,\bsb]$ and asking the function $f$ and its derivatives to decay towards infinity.
We consider two decay conditions: one for which the function and its derivatives decay exponentially fast, and one for which the decay is only polynomial. Both conditions are quite natural if one assumes the integrand to be the product of a function with a probability density. E.g., we would have exponential decay in this sense for any polynomial multiplied with the normal density.
From the construction of $F^{[\bsa,\bsb]}$ we can then show pointwise convergence to $f$.

\begin{proposition}\label{prop:projection-error}
 Let $F^{[\bsa,\bsb]}$ be the periodic function on the box $[\bsa,\bsb]$ obtained by applying the Bernoulli polynomial method on the box~\eqref{eq:BPM-ab} to a function $f$ on the box $[\bsa,\bsb]$ having mixed partial derivatives up to order $\alpha-1$ in each variable which are absolutely continuous and of which the derivatives with the highest order $\alpha$ are integrable. Then the following hold true.
 \\
 \ri If $f$ satisfies the exponential decay condition
  \begin{align}\label{eq:exp-decay}
    \|f\|_{\alpha,\beta,p,q}
    &:=
    \sup_{\substack{\bsx\in\R^d \\ \bstau\in\{0,\ldots,\alpha-1\}^d}}
     \left| \exp(\beta \, \|\bsx\|_p^q) \, f^{(\bstau)}(\bsx) \right|
    <
    \infty
    ,
  \end{align}
  for some $\beta > 0$, $1 \le p \le \infty$ and $1 \le q < \infty$, then the projection error can be bounded for all $\bsx \in [\bsa,\bsb]$ by
  \begin{align*}
    \left| F^{[\bsa,\bsb]}(\bsx) - f(\bsx) \right|
    &\le
    (\alpha+1)^d \left[ \prod_{j=1}^d \max(1, b_j-a_j)^{\alpha-1} \right]
    \,
    \|f\|_{\alpha,\beta,p,q}
    \,
    \exp(-\beta \min_{1 \le j \le d}\min\{|a_j|^q,|b_j|^q\})
    .
  \end{align*}
  \\
  \rii If $f$ satisfies the polynomial decay condition
  \begin{align}\label{eq:poly-decay}
   \|f \|_{\alpha,\beta,p} 
 :=
 \sup_{\substack{\bsx\in\R^d \\ \bstau\in\{0,\ldots,\alpha-1\}^d }} \left| \|\bsx\|_p^{\beta} \, f^{(\bstau)}(\bsx) \right|
 <
 \infty,
  \end{align}
  for some $\beta > 0$ and  $1 \le p \le \infty$, then the projection error can be bounded for all $\bsx \in [\bsa,\bsb]$ by
  \begin{align*}
    \left| F^{[\bsa,\bsb]}(\bsx) - f(\bsx) \right|
    &\le
    (\alpha+1)^d \left[ \prod_{j=1}^d \max(1, b_j-a_j)^{\alpha-1} \right]
    \,
    \|f\|_{\alpha,\beta,p}
    \,
    \left(\min_{1 \le j \le d}\min\{|a_j|,|b_j|\}\right)^{-\beta}
    .
  \end{align*}
\end{proposition}
\begin{proof}
  \ri
  Using the definition of $F^{[\bsa,\bsb]}$, see~\eqref{eq:BPM-ab}, the bound on the maximum magnitude of the Bernoulli polynomials~\eqref{eq:BerBound} and the decay condition~\eqref{eq:exp-decay} we obtain
  \begin{align*}
    &\left| F^{[\bsa,\bsb]}(\bsx) - f(\bsx) \right|
    \\
    &\quad=
    \left|
    \sum_{\substack{\bszero \ne \bstau \in \{0,\ldots,\alpha\}^d \\ \setu := \supp(\bstau)}}
      (-1)^{|\setu|}
      \left[ \prod_{j \in \setu} \frac{B^{[a_j,b_j]}_{\tau_j}(x_j)}{\tau_j!} \right]
        \sum_{\setw\subseteq\setu} (-1)^{|\setw|}
          f^{(\bstau_\setu-\bsone_\setu, \bszero_{-\setu})}(\bsa_\setw,\bsb_{\setu\setminus\setw}, \bsx_{-\setu})
    \right|
    \\
    &\quad\le
    \sum_{\substack{\bszero \ne \bstau \in \{0,\ldots,\alpha\}^d \\ \setu := \supp(\bstau)}}
      \left[ \prod_{j \in \setu} \frac{(b_j-a_j)^{\tau_j-1}}{2} \right]
      2^{|\setu|} \,
      \|f\|_{\alpha,\beta,p,q} \,
      \max_{\bsx \in \delta([\bsa,\bsb])} \exp(-\beta \, \|\bsx\|_p^q)
    \\
    &\quad\le
    \sum_{\substack{\bszero \ne \bstau \in \{0,\ldots,\alpha\}^d \\ \setu := \supp(\bstau)}}
      \left[ \prod_{j \in \setu} \frac{(b_j-a_j)^{\tau_j-1}}{2} \right]
      2^{|\setu|} \,
      \|f\|_{\alpha,\beta,p,q} \,
      \max_{\bsx \in \delta([\bsa,\bsb])} \exp(-\beta \, \|\bsx\|_\infty^q)
    \\
    &\quad\le
    \sum_{\substack{\bszero \ne \bstau \in \{0,\ldots,\alpha\}^d \\ \setu := \supp(\bstau)}}
      \left[ \prod_{j \in \setu} \frac{(b_j-a_j)^{\tau_j-1}}{2} \right]
      2^{|\setu|} \,
      \|f\|_{\alpha,\beta,p,q} \, \exp(-\beta \min_{j\in\setu}\min\{|a_j|^q,|b_j|^q\})
      ,
  \end{align*}
  where we used that the point $(\bsa_\setw,\bsb_{\setu\setminus\setw}, \bsx_{-\setu})$ is always on the boundary of the box $[\bsa,\bsb]$, which we denoted by $\delta([\bsa,\bsb])$ in the derivation above, since $\setu \ne \emptyset$.
  Next we used that $\|\cdot\|_\infty^q \le \|\cdot\|_p^q$ for $p, q \ge 1$.
  From here the result follows.
  \\
  \rii For the second claim the proof is similar.
\end{proof}

As we mentioned in Section~\ref{sec:intro}, the remaining part of the total error is the truncation error 
\[
   \left|
    \int_{\R^d} f(\bsx) \rd \bsx
    -
    \int_{[\bsa,\bsb]} f(\bsx) \rd \bsx
  \right|.
\]
We are going to bound this truncation error where the appropriate truncation box is given depending on the decay condition.
We will propose to use a box $[-a,a]^d$ with a single parameter~$a > 0$.
If the user knows more about the anisotropy of the integrand's decay, or if the decay is off-center, then this knowledge can be used to recenter and rescale the decay in the different directions.

\begin{proposition}\label{prop:truncation}
\ri
Suppose an integrand function $f$ satisfies the exponential decay condition~\eqref{eq:exp-decay} for $\alpha = 1$,
\[
 \|f\|_{1,\beta,p,q}
 =
 \sup_{\bsx\in\R^d}
 \left| \exp(\beta \, \|\bsx \|_p^q) \, f(\bsx) \right|
 <
 \infty
 ,
\]%
for some $\beta > 0$, $1 \le p \le \infty$ and $1 \le q < \infty$,
then, by choosing the box to be $[-a,a]^d$, for some $a > 0$,
the truncation error is bounded for any suitable $\alpha \in \N$ by
\begin{align*}
  &
  \left|
    \int_{\R^d} f(\bsx) \rd \bsx
    -
    \int_{[-a,a]^d} f(\bsx) \rd \bsx
  \right|
  \le
  \frac{2^d \, d}{\beta^{d/q} \, q} \, \lceil d/q \rceil! \,
  \|f\|_{\alpha,\beta,p,q} \, \frac{\exp(-\beta \, a^q)}{\beta \, a^q} \, \max\{1, (\beta \, a^q)^{d/q}\}
  .
\end{align*}
\\
\rii
Suppose an integrand function $f$ satisfies the polynomial decay condition~\eqref{eq:poly-decay} for $\alpha = 1$,
\[
 \|f \|_{1,\beta,p}
 =
 \sup_{\bsx\in\R^d} \left| \|\bsx\|_p^{\beta } \, f(\bsx) \right|
 <
 \infty,
\]%
for some $\beta > d$ and $1 \le p \le \infty$,
then, by choosing the box
to be $[-a, a]^d$, for some $a > 0$,
the truncation error is bounded for any suitable $\alpha \in \N$ by
\begin{align*}
  &
  \left|
    \int_{\R^d} f(\bsx) \rd \bsx
    -
    \int_{[-a,a]^d} f(\bsx) \rd \bsx
  \right|
  \le
  \frac{2^d \, d}{\beta-d} \,
  \|f\|_{\alpha,\beta,p} \, a^{-\beta+d}
  .
\end{align*}
\end{proposition}
\begin{proof}
\ri We have
 \begin{align*}
  \left|
    \int_{\R^d} f(\bsx) \rd \bsx
    -
    \int_{[\bsa,\bsb]} f(\bsx) \rd \bsx
  \right|
  &\le
  \|f\|_{\alpha,\beta,p,q} \, \int_{\R^d \setminus [\bsa,\bsb]} \exp(-\beta \, \|\bsx\|_p^q)  \rd \bsx 
  \\
  &\le
  \|f\|_{\alpha,\beta,p,q} \, \int_{\R^d \setminus [\bsa,\bsb]} \exp(-\beta \, \|\bsx\|_{\infty}^q)  \rd \bsx
  \\
  &\le
  \|f\|_{\alpha,\beta,p,q} \, \int_{\R^d \setminus [-a,a]^d} \exp(-\beta \, \|\bsx\|_{\infty}^q)  \rd \bsx
  ,
\end{align*}
where we have set $a = \min_{1 \le j \le d} \min\{|a_j|, |b_j|\}$.
By setting $r = \|\bsx\|_{\infty}$ and integrating from $a$ to $\infty$, the volume element is $2^d \, d \, r^{d-1} \rd r$, and hence
\begin{align*}
  \int_{\R^d \setminus [-a,a]^d} \exp(-\beta \, \|\bsx\|_{\infty}^q)  \rd \bsx
  &=
  \int_a^\infty \exp(-\beta \, r^q) \, 2^d \, d \, r^{d-1} \rd r
  \\
  &=
  \frac{2^d \, d}{\beta^{d/q} \, q} \,
  \int_{\beta a^q}^\infty \exp(-t) \, t^{(d/q)-1} \rd t
  \\
  &=
  \frac{2^d \, d}{\beta^{d/q} \, q} \, \Gamma(d/q, \beta a^q)
  ,
\end{align*}
where $\Gamma(s, z)$ is the upper incomplete Gamma function with $s = d/q > 0$ and $z = \beta a^q > 0$.
From \cite[\S8.8.2]{NIST:DLMF} we can use recursion to reduce the first argument of the incomplete Gamma function to be in the interval $(0, 1]$, i.e.,
\begin{align*}
  \Gamma(s, z)
  &=
  \Gamma(s - \lfloor s \rfloor, z) \left[ \prod_{i=1}^{\lfloor s \rfloor} (s-i) \right]
  +
  \exp(-z) \, \sum_{j=1}^{\lfloor s \rfloor} \left[ \prod_{i=1}^{j-1} (s-i) \right] z^{s - j}
  .
\end{align*}
Next we use \cite[\S8.10.1]{NIST:DLMF} which states that $\Gamma(s', z) \le z^{s'-1} \, \exp(-z)$ for $0 < s' \le 1$ and $z > 0$ to obtain
\begin{align*}
  \Gamma(d/q, \beta a^q)
  =
  \Gamma(s, z)
  &\le
  \exp(-z)  \left[ \prod_{i=1}^{\lfloor s \rfloor} (s-i) \right] z^{s - \lfloor s \rfloor - 1}
  +
  \exp(-z) \, \sum_{j=1}^{\lfloor s \rfloor} \left[ \prod_{i=1}^{j-1} (s-i) \right] z^{s - j}
  \\
  &=
  \exp(-z) \, \sum_{j=0}^{\lfloor s \rfloor} \left[ \prod_{i=1}^j (s-i) \right] z^{s - j - 1}
  =
  \frac{\exp(-z)}{z} \, \sum_{j=0}^{\lfloor s \rfloor} \left[ \prod_{i=1}^j (s-i) \right] z^{s - j}
  \\
  &\le
  \frac{\exp(-z)}{z} \, \lceil s \rceil! \, \max\{1, z^s\}
  ,
\end{align*}
which follows by considering the cases $s \in \N$, for which $\prod_{i=1}^{\lfloor s \rfloor} (s - i) = \prod_{i=1}^s (s - i) = 0$ and so we have $\lfloor s \rfloor$ terms in the sum and $\lfloor s \rfloor \, \lceil s - 1 \rceil! \le \lceil s \rceil!$, and $s \notin \N$, for which $\lfloor s \rfloor + 1 \le \lceil s \rceil$, separately.
\\
\rii Similar to case \ri we have, for $\beta > d$,
\begin{align*}
  \left|
    \int_{\R^d} f(\bsx) \rd \bsx
    -
    \int_{[\bsa,\bsb]} f(\bsx) \rd \bsx
  \right|
  &\le
  \|f\|_{\alpha,\beta,p}  \, \int_{\R^d \setminus [\bsa,\bsb]} \|\bsx\|_p^{-\beta}   \rd \bsx 
  \\
  &\le
   \|f\|_{\alpha,\beta,p} \, \int_{\R^d \setminus [\bsa,\bsb]} \|\bsx\|_{\infty}^{-\beta}  \rd \bsx 
  \\
  &\le
  \|f\|_{\alpha,\beta,p}  \, \int_a^\infty r^{-\beta} \, 2^d \, d \, r^{d-1} \rd r 
  \\
  &=
  \|f\|_{\alpha,\beta,p}  \, 2^d \, d \, \frac{a^{-\beta+d}}{\beta-d}
  . \qedhere
\end{align*}
\end{proof}

To make the statement for integration over $\R^d$ we need to know how the norm in $\calH(\KSobab_{\alpha,d})$ of our integrand function behaves as the box grows larger.
Since our inner product~\eqref{eq:Sob-ab-innerproduct} is an ``unanchored'' one, meaning that we first integrate out all dimensions in which we take derivatives lower than the smoothness $\alpha$ of the space, we cannot avoid that for certain boxes our norm becomes very small (e.g., some of these integrals might vanish for non-zero functions).
This would be different with a more classical norm in which this inner integration is not present, and the outer $L_2$ integral integrates over all dimensions.
In such a case the norm can only grow bigger when the box increases.
To be able to state results for arbitrary boxes using our unanchored norm leads us to define a norm which captures the worst possible box as follows
\begin{align}\label{eq:Rd-norm}
 \|f\|_{\KSobstar_{\alpha,d}} 
 &:=
 \sup_{[\bsa,\bsb] \subset \R^d} \|f\|_{\KSobab_{\alpha,d}}
 .
\end{align}
We can now make the main statement of the paper.

\begin{theorem}\label{thm:total-error}
Suppose $f \in \calH(\KSobab_{\alpha,d})$ for increasing boxes $[\bsa,\bsb]$, i.e., 
 $\|f\|_{\KSobstar_{\alpha,d}} < \infty$, cf.~\eqref{eq:Rd-norm}. \\
\ri If additionally, $f$ satisfies the exponential decay condition~\eqref{eq:exp-decay},
\[
 \|f\|_{\alpha,\beta,p,q} 
 =
 \sup_{\substack{\bsx\in\R^d \\ \bstau\in\{0,\ldots,\alpha-1\}^d}}
 \left| \exp(\beta \, \|\bsx\|_p^q) \, f^{(\bstau)}(\bsx) \right|
 <
 \infty
 ,
\]
for some $\beta > 0$, $1 \le p \le \infty$ and $1 \le q < \infty$,
then, by choosing the box proportional to $[-a,a]^d$ with
\begin{align*}
  a
  &=
  (\log(n^{\alpha/\beta}))^{1/q}
  ,
\end{align*}
and using a good lattice rule $\Lambda(\bsz,n)$ for the unit cube $[0,1]^d$ as per Lemma~\ref{lem:cuberror-ab}, the total error for $n$ large enough is bounded by
\begin{align*}
  \left|
    \int_{\R^d} f(\bsx) \rd \bsx
    -
    \frac{\prod_{j=1}^d (b_j-a_j)}{n} \sum_{i=1}^n f(\bsp_i^{[\bsa,\bsb]})
  \right|
  \le
  C \, \left(\|f\|_{\KSobstar_{\alpha,d}} + \|f\|_{\alpha,\beta,p,q}\right) \,
  \frac{(\log n)^{\alpha d + (\alpha + 1/2) \, d/q}}{n^\alpha}
  ,
\end{align*}
where the constant $C$ is independent of $n$ and $f$ but depending on $\alpha$, $\beta$, $q$ and $d$.
\\
\rii If additionally, $f$ satisfies the polynomial decay condition~\eqref{eq:poly-decay},
\[
  \|f \|_{\alpha,\beta,p} 
  =
  \sup_{\substack{\bsx\in\R^d \\ \bstau\in\{0,\ldots,\alpha-1\}^d }} \left| \|\bsx\|_p^{\beta} \, f^{(\bstau)}(\bsx) \right|
  <
  \infty,
\]
for some $\beta > d \, \max\{\alpha-1, 1\}$ and $1 \le p \le \infty$,
then, by choosing the box proportional to $[-a,a]^d$ with
\begin{align*}
  a
  &=
  n^{\alpha/(\beta+t d/2)}
\end{align*}
with $t = 3$ for $\alpha \ge 2$ and $t = 1$ for $\alpha = 1$,
and using a good lattice rule $\Lambda(\bsz,n)$ for the unit cube $[0,1]^d$ as per Lemma~\ref{lem:cuberror-ab}, the total error for $n$ large enough is bounded by
\begin{align*}
  \left|
    \int_{\R^d} f(\bsx) \rd \bsx
    -
    \frac{\prod_{j=1}^d (b_j-a_j)}{n} \sum_{i=1}^n f(\bsp_i^{[\bsa,\bsb]})
  \right|
  \le
  C \, \left(\|f\|_{\KSobstar_{\alpha,d}} +\|f\|_{\alpha,\beta,p}\right)
  \, n^{-\alpha + \frac{\alpha (1+2\alpha)}{t + 2\beta / d}}
  \, (\log n)^{\alpha d}
  ,
\end{align*}
where the constant $C$ is independent of $n$ and $f$ but depending on $\alpha$, $\beta$ and $d$.
\end{theorem}
\begin{proof}
The total error of our approximation can be considered the sum of three errors, cf.~\eqref{eq:split-error-analysis}, which we can bound by using Proposition~\ref{prop:truncation} for the truncation error, using Proposition~\ref{prop:lattice-rule-error-ab} and Lemma~\ref{lem:cuberror-ab} for the cubature error of the periodic function $F^{[\bsa,\bsb]}$, and finally using Proposition~\ref{prop:projection-error} for the projection error between $F^{[\bsa,\bsb]}$ and~$f$.

\ri For the exponential decay and truncating to a box $[\bsa,\bsb] = [-a,a]^d$ with $a$ chosen such that $\beta a^q = \log(n^\alpha)$, for $n \ge 3$ and $\log(n^\alpha) \ge \beta / 2^q$ such that $a \ge 1/2$, we obtain
\begin{align*}
  &\left| \int_{\R^d} f(\bsx) \rd \bsx  - \frac{\prod_{j=1}^d (b_j-a_j)}{n}\sum_{i=1}^n f(\bsp_i^{[\bsa,\bsb]}) \right|
  \\
  &\qquad\le
  C_1 \, \|f\|_{\alpha,\beta,p,q} \, \max\{1, (\beta \, a^q)^{d/q}\} \, \frac{\exp(-\beta \, a^q)}{\beta \, a^q}
  \\
  &\qquad\quad+
  C_2 \, \|f\|_{\KSobab_{\alpha,d}}
    \,
    \max\{1, 2a\}^{(\alpha+1/2) \, d} \,
    \frac{(\log n)^{\alpha d}}{n^\alpha}
  \\
  &\qquad\quad+
  C_3
    \,
    \|f\|_{\alpha,\beta,p,q}
     \, \max\{1, 2a\}^{(\alpha-1) \, d}
    \,
    \exp(-\beta a^q) \vphantom{\frac{1}{\beta n^\alpha}}
  \\
  &\qquad\le
  C \, \left( \|f\|_{\alpha,\beta,p,q} + \|f\|_{\KSobstar_{\alpha,d}} \right) \,
  (\log n)^{\alpha d + (\alpha + 1/2) \, d / q} \, n^{-\alpha}
  .
\end{align*}

\rii For the polynomial decay and truncating to a box $[\bsa,\bsb] = [-a,a]^d$ with $a = n^{\alpha/(\beta+3d/2)}$, and assuming $n$ is large enough such that $a \ge 1/2$, we obtain by equalizing the exponent of $n$ for the cubature error and the projection error with $\alpha \ge 2$,
\begin{align*}
  &\left| \int_{\R^d} f(\bsx) \rd \bsx  - \frac{\prod_{j=1}^d (b_j-a_j)}{n}\sum_{i=1}^n f(\bsp_i^{[\bsa,\bsb]}) \right|
  \\
  &\le
  C_1 \, \|f\|_{\alpha,\beta,p} \, a^{-\beta+d}
  \\
  &\quad+
  C_2 \, 
  \|f\|_{\KSobab_{\alpha,d}}
    \,
    \max\{1, 2a\}^{(\alpha+1/2) \, d} \,
    \frac{(\log n)^{\alpha d}}{n^\alpha}
  \\
  &\quad+
  C_3 \,
  \|f\|_{\alpha,\beta,p}
    \,
  \max\{1, 2a\}^{(\alpha-1) \, d}
    \,
    a^{-\beta} \vphantom{\frac{1}{\beta n^\alpha}}
  \\
  &\le
  C \, \left( \|f\|_{\alpha,\beta,p} + \|f\|_{\KSobstar_{\alpha,d}} \right) \,
  \max\{
    n^{\frac{-\alpha}{1+3d/(2\beta)} + \frac{\alpha}{3/2+\beta/d}} ,
    n^{-\alpha + \frac{\alpha (\alpha+1/2)}{3/2+\beta/d}} ,
    n^{\frac{-\alpha}{1+3d/(2\beta)} + \frac{\alpha (\alpha-1)}{3/2+\beta/d}}
  \}
  \, (\log n)^{\alpha d}
  \\
  &\le
  C \, \left( \|f\|_{\alpha,\beta,p} + \|f\|_{\KSobstar_{\alpha,d}} \right) \,
    n^{-\alpha + \frac{\alpha (1+2\alpha)}{3 + 2\beta / d}}
  \, (\log n)^{\alpha d}
  .
\end{align*}
For this to converge we need $\beta > (\alpha - 1) \, d$.
When $\alpha = 1$ we instead equalize the truncation error with the cubature error and by taking $a = n^{1/(\beta+d/2)}$ we obtain in the same manner
\begin{multline*}
  \left| \int_{\R^d} f(\bsx) \rd \bsx  - \frac{\prod_{j=1}^d (b_j-a_j)}{n}\sum_{i=1}^n f(\bsp_i^{[\bsa,\bsb]}) \right|
  \\\le
  C \, \left( \|f\|_{1,\beta,p} + \|f\|_{\KSobstar_{1,d}} \right) \,
  n^{-1 + \frac{3}{1+2\beta/d}} \,
  \, (\log n)^{\alpha d}
  ,
\end{multline*}
with the condition $\beta > d$.
\end{proof}

We want to remark that when the considered integrand has slow decay, such as calculating expectations under heavy-tailed distributions (e.g., the Student-t distribution and the alpha-stable distribution with stable parameter smaller than 2), the method here proposed still works. Those problems especially occur in finance and extreme value theory. 
For such distributions the classical Gaussian-type quadrature cannot be constructed for arbitrary number of nodes since such distributions do not have higher-order moments.

\section{Numerical experiments}\label{sec:numerics}

We here demonstrate our method on two test cases, but first we want to draw the reader's attention to the fact that our algorithm is indeed very simple.
We do not require any inverse sampling, nor any complicated method to turn our point sets into higher-order point sets, or any periodizing transform or complicated adjustments involving calculating derivatives.
We simply need the user to assess the decay of the integrand function, which is mostly just dictated by the decay of the actual distribution, as is the case in our examples here.
E.g., any polynomially increasing function integrated against the standard normal distribution will fit the exponential decay with $q = 2$ and $\beta = 1/2$.
Similarly, any polynomially increasing function integrated against the logistic density with scale parameter $s > 0$ will fit the exponential decay with $q = 1$ and $\beta = 1/s$.
These parameters will let us calculate the box in terms of $n$ and then we do a simple scaling of the lattice rule to this box.
The pseudo code is given in Algorithm~\ref{alg:scaled-lattice-rule}.
Note that if information on the anisotropy of the integrand function is known (e.g., because of different scaling or variance parameters), or if the function is off-centre, then this information could be used to improve the selection of the box such that convergence kicks in sooner.
For ease of exposition we do not discuss that here.

Although in the analysis we have split the error into three contributions, cf.~\eqref{eq:split-error-analysis}, for our numerical experiments it makes more sense to look at~\eqref{eq:split-error-numerics} which splits the error in the two contributions which are apparent from the algorithm: the error of truncating to the box and the error of approximating the integral on the box by the lattice rule.
In the analysis we have split the last part into two parts by considering the projection $F^{[\bsa,\bsb]}$, but this was just used as a theoretical tool which we do not have available in the algorithm (or, in fact we avoid it because of its excessive cost).

Since our error can thus be considered as the sum of the two contributions in~\eqref{eq:split-error-numerics} we need to make sure that we have an idea of both of these terms, since one could dominate the other, and then we do not know if indeed both of the terms converge as we have claimed in the theoretical analysis, cf.~Proposition~\ref{prop:truncation} for the truncation error and Lemma~\ref{lem:cuberror-ab} for the cubature error on the box.
To facilitate this we will use products of one-dimensional functions of which we can obtain the true value up to machine precision for the one-dimensional integrals over $\R$ and over a box $[a,b]$.
We choose one-dimensional functions of the form $(1 + g_j(x_j)) \, \rho_j(x_j)$ such that the product $\prod_{j=1}^d ((1 + g_j(x_j)) \, \rho_j(x_j)) = (\sum_{\mathfrak{u} \subseteq \{1,\ldots,d\}} \prod_{j\in\mathfrak{u}} g_j(x_j)) \, \prod_{j=1}^d \rho_j(x_j)$ contains interactions of variables in all possible subsets of $\{1,\ldots,d\}$ against a $d$-dimensional product density $\prod_{j=1}^d \rho_j(x_j)$.
We use an automatic quadrature routine to calculate the one-dimensional reference values if needed to calculate the error contributions.
We consider integrands with finite smoothness and compare our method to other existing methods.

\begin{algorithm}[t]
\caption{Scaled lattice rules for integration on $\R^d$ by truncating to a box $[\bsa,\bsb]$}\label{alg:scaled-lattice-rule}
\begin{algorithmic}
\State{\textbf{Input:}}
\State{- $f$} \Comment{$d$-dimensional integrand function}
\State{- $\alpha$} \Comment{smoothness parameter of Sobolev space such that $\|f\|_{\KSobstar_{\alpha,d}} < \infty$}
\State{- decay condition: $\beta$, $p$, $q$} \Comment{parameters for decay condition \ri or \rii}
\State{- $\bsz, n$} \Comment{generating vector of good $n$-point $d$-dimensional lattice rule}
\Statex{}
\State{$[\bsa, \bsb] := [-a,a]^d$} \Comment{calculate using Theorem~\ref{thm:total-error} for decay \ri or \riip, given $n$, $\alpha$, $\beta$, $q$, $d$}
\State{$Q := 0$}
\For{\texttt{$i=1,2,...,n$}} \Comment{possibly in parallel}
        \State{$Q := Q + \frac{\prod_{j=1}^d (b_j-a_j)}{n} f(\bsp^{[\bsa,\bsb]}_i)$}
        \Comment{with $p^{[\bsa,\bsb]}_{i,j} = (b_j-a_j) \, \left((i \bsz / n) \bmod 1\right) + a_j$}
\EndFor
\Statex{}
\State \textbf{Output:} $Q$
\end{algorithmic}
\end{algorithm}

\subsection{Integration against the logistic distribution}

We consider the following function, being finitely smooth with parameter $\sigma \in \R \setminus \Z$ and $\sigma > 0$:
\begin{align*}
  f_1(\bsx)
  &:=
  \prod_{j=1}^d \left[
    \left(1 +  4x_j + 10 \cos^2(x_j) + \operatorname{sign}(x_j-\mu_j) \, \frac{|x_j-\mu_j|^\sigma}{\Gamma(\sigma+1)}\right)
    \, \frac{\exp((x_j-\mu_j)/s_j)}{(1+\exp((x_j-\mu_j)/s_j))^2 \, s_j} \right]
  .
\end{align*}
We note that $\operatorname{sign}(x) \, |x|^\sigma \in \KSobabone_{\alpha,1}$ for all $\alpha < \sigma + 1/2$, and thus $f_1 \in \KSobstar_{\alpha,d}$ for $\alpha < \sigma + 1/2$, which follows from the Leibniz rule for the product of this function with the logistic probability density function.
Since we are integrating a polynomially increasing function against the logistic distribution this integrand satisfies the exponential decay condition~\eqref{eq:exp-decay} with parameters $q = 1$ and $\beta = 1/s$ where $s = \max_{1 \le j \le d} s_j$.
The value of the integral is given by
\begin{align*}
  \int_{\R^d} f_1(\bsx) \rd\bsx
  &=
  \prod_{j=1}^d \left(1 + 4 \mu_j + \frac{10}{2} + 10 \pi \, s_j \frac{\cos(2 \mu_j)}{\sinh(2 \pi s_j)} \right)
  .
\end{align*}
For $d = 2$ we consider $\mu = (3, -3)$ and $s = (2, 2)$.
For $d = 3$ we consider $\mu = (1, -1, 0)$ and $s = (1, 1, 1)$.
We choose three different values $\sigma = 0.6$, $1.6$ and $2.6$ which respectively translate to smoothness $\alpha = 1$, $2$ and~$3$.

We apply Algorithm~\ref{alg:scaled-lattice-rule} without recentering the integrand and we compare to a very similar algorithm which uses scaled higher-order digital nets from \cite{DILP2018} which was developed for integration against the normal distribution.
The algorithm in \cite{DILP2018} takes the same form as our algorithm here, with the difference that the analysis does not need to take care of the projection error since higher-order digital nets are directly applicable to non-periodic spaces.
Although the analysis in \cite{DILP2018} is limited for integration against the normal distribution, the basic reasoning followed in Theorem~\ref{thm:total-error}, without the projection error, also holds for other higher-order cubature rules on the cube, and hence we can apply scaled higher-order digital nets for the integrand here as well.
We use interlaced Sobol' points from \cite{KN2016} with the interlacing factor being the smoothness $\alpha$ of the considered integrand, i.e., $\alpha= \lfloor \sigma + 1/2 \rfloor$. 
For the lattice rule the generating vector is always fixed to be $\bsz = (1, 4959637, 5860107)$ for any number of points $n$ being a power of two.
When the dimension $d$ is smaller than three, the generating vector simply becomes the first $d$ components.
This generating vector was constructed for a lattice sequence in base~$2$ by the component-by-component construction algorithm for the first order unweighted Korobov space on the unit cube with the numbers of points from $2^8$ to $2^{24}$, see~\cite{CKN2006}.
This could be considered an off-the-shelf lattice sequence as we did not construct separate lattice sequences for different orders of smoothness.
For both the interlaced Sobol' sequence and the lattice sequence we need to restrict the number of points to be a power of~$2$ to see the higher-order convergence, see~\cite{HKKN2012}.

The result is exhibited in Figures~\ref{fig:numeric-f1d2} and~\ref{fig:numeric-f1d3}.
We expect both methods to achieve order $\alpha$ convergence, but only the scaled lattice rule does so convincingly.
The scaled interlaced Sobol' sequence seems to not yet be in its asymptotical regime and is only slightly better than order~$1$ for $\alpha=2$ and~$3$, leaving a huge gap with the scaled lattice rule.
We notice from the figures that when the dimension goes up from $2$ to $3$ that the asymptotic regime also kicks in later for the scaled lattice rule, but this is to be expected since we did not try to accommodate weighted function spaces.
For $d=3$ and higher smoothness it is also apparent that the error is dominated by the error of the cubature rule on the box.
To balance this error one would need to have an automatic method which can assess both errors separately.

\begin{figure}
\centering
\includegraphics[scale=1]{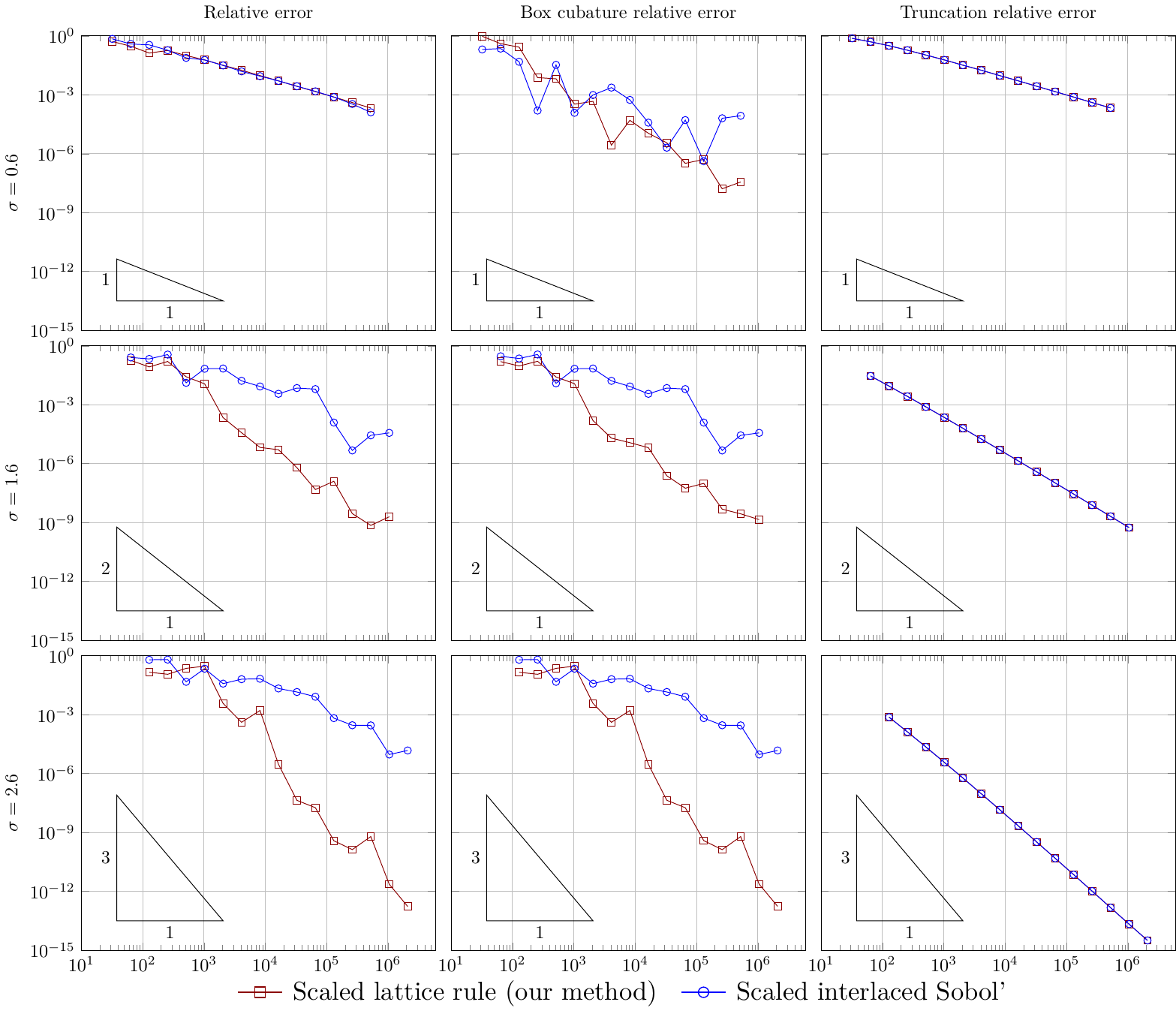}
 \caption{Relative errors for $f_1$ with $d = 2$.}
 \label{fig:numeric-f1d2}
\end{figure}

\begin{figure}
\centering
\includegraphics[scale=1]{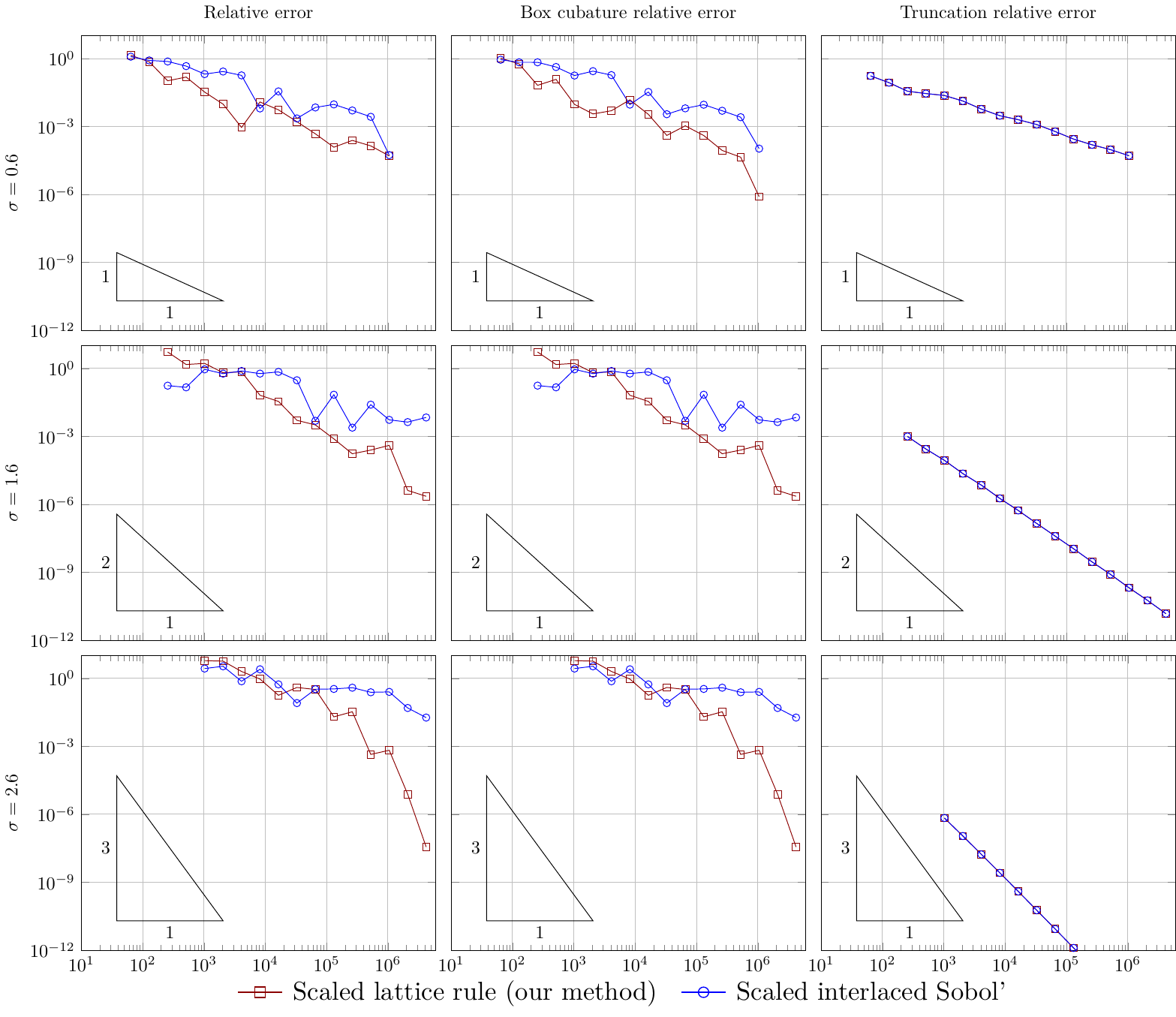}
 \caption{Relative errors for $f_1$ with $d = 3$.}
 \label{fig:numeric-f1d3}
\end{figure}

\subsection{Integration against the normal distribution}

Next, we consider the following integrand with $\sigma \ge 0$:
\begin{align*}
  f_2(\bsx)
  &:=
  \frac{\exp(-\|\bsx\|^2_2/2)}{(2\pi)^{d/2}} \, \prod_{j=1}^d  \left( 1 + |x_j|^\sigma \right)
  ,
\end{align*}
which is basically a product of moments against the normal distribution and has exact value:
\begin{align*}
  \int_{\R^d} f_2(\bsx) \rd \bsx
  =
  \left( 1 + \frac{(\sqrt{2})^\sigma \, \Gamma((\sigma+1)/2)}{\sqrt{\pi}} \right)^d
  .
\end{align*}
For $\sigma \in \R \setminus \Z$ and $\sigma > 0$ we again have that $|x|^\sigma$ has finite smoothness and $|x|^\sigma \in \KSobabone_{\alpha,1}$ for all $\alpha < \sigma + 1/2$, and thus $f_2 \in \KSobstar_{\alpha,d}$ for $\alpha < \sigma + 1/2$, which follows from the Leibniz rule for the product of this function with the Gaussian probability density function.
This integrand satisfies the exponential decay condition~\eqref{eq:exp-decay} with parameters $q = 2$ and $\beta = 1/2$.
We choose three different values $\sigma = 0.6$, $1.6$ and $2.6$ which respectively translate to smoothness $\alpha = 1$, $2$ and~$3$ in $d=2$ and $3$ dimensions.

We compare our method with a tensor product Gauss--Hermite quadrature with the number of points being $2^m+1$ in each dimension with $m \in \N$, a sparse grid based on the same one-dimensional $2^m+1$ point Gauss--Hermite rules and the scaled interlaced Sobol' points as described in the previous section.
For the sparse grid we use the Sparse Grids Matlab Kit \cite{BNTT2011} and use the default Smolyak construction.
We remark that neither the tensor product rule, nor the sparse grid explicitly truncate the domain.

The result is exhibited in Figures~\ref{fig:numeric-f2d2} and~\ref{fig:numeric-f2d3}.
Again, we expect that our method and the scaled interlaced Sobol' points achieve  order $\alpha$ convergence.
Contrary to the result for $f_1$, here the scaled interlaced Sobol' points do give the expected higher order convergence, but they still underperform and leave a gap with the scaled lattice rules for the higher order, which seems to increase as the dimension or the smoothness goes up.
For the tensor product Gauss--Hermite rule it was already observed in \cite{DILP2018} that the convergence was not optimal, and also here we observe this behavior.
The sparse grid performs better than the tensor product rule, as expected, and seems to catch up with the performance of the scaled interlaced Sobol' points.
Nevertheless, in all cases, the scaled lattice rule outperforms all other methods and achieves the claimed convergence rates in accordance with the theory.

\begin{figure}
\centering
\includegraphics[scale=1]{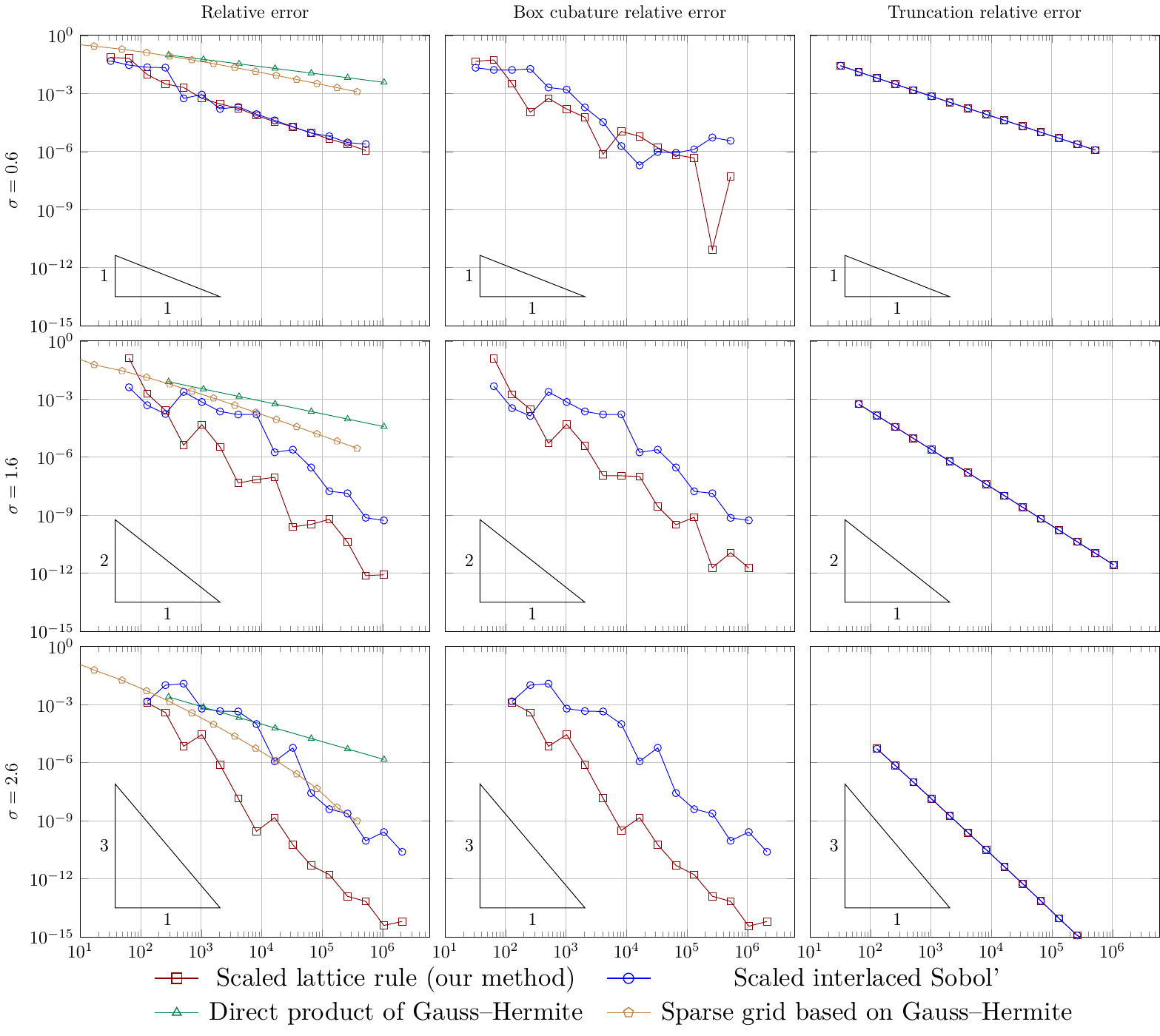}
 \caption{Relative errors for $f_2$ with $d = 2$.}
 \label{fig:numeric-f2d2}
\end{figure}

\begin{figure}
\centering
\includegraphics[scale=1]{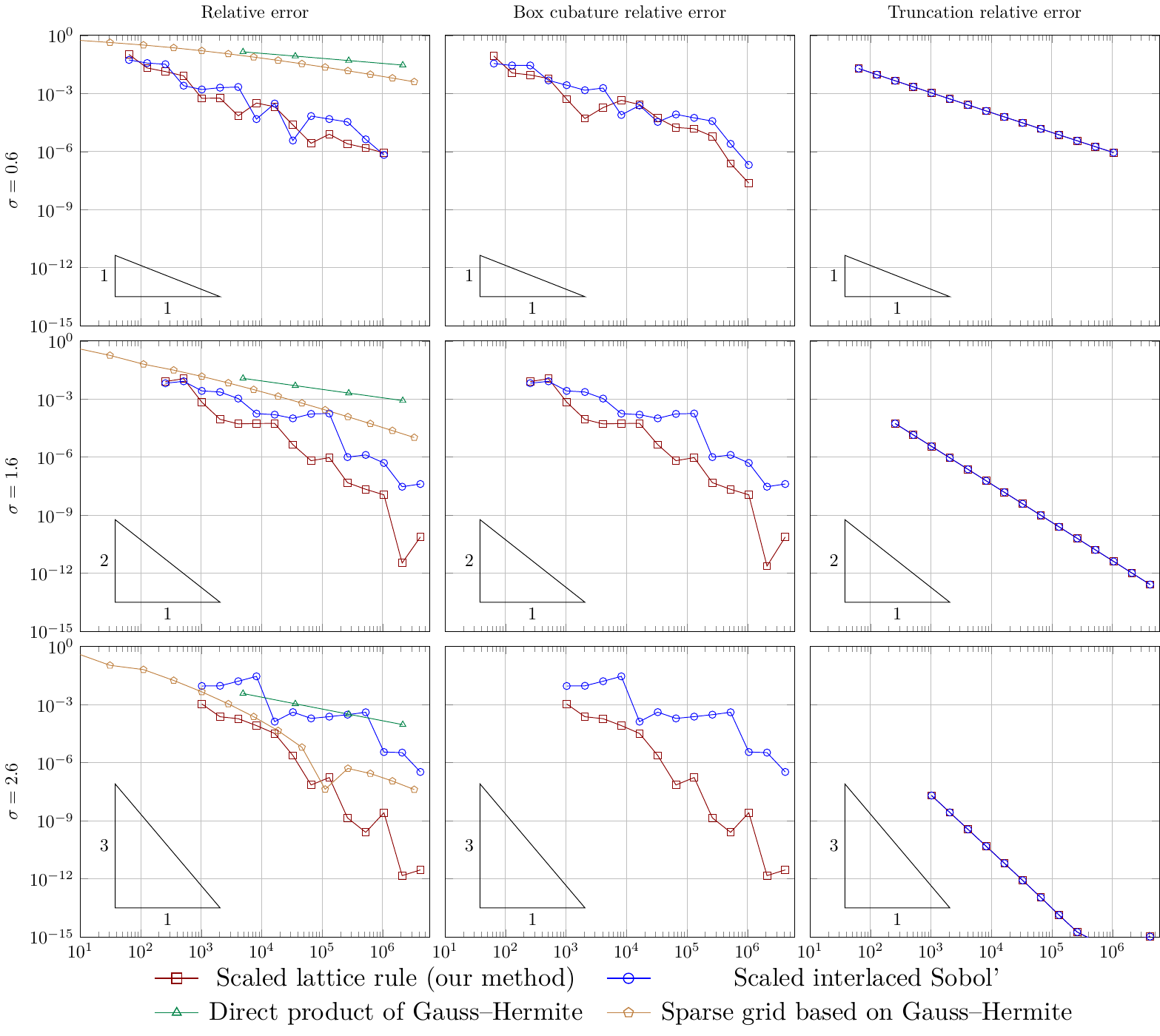}
 \caption{Relative errors for $f_2$ with $d = 3$.}
 \label{fig:numeric-f2d3}
\end{figure}

\section{Conclusion}\label{sec:conclusion}

We considered a very simple algorithm based on scaled lattice rules for integration over $\R^d$.
We derived explicit conditions where scaled lattice rules can obtain higher-order convergence.
In our new strategy of the error analysis, the total error is divided into three parts: the truncation error, the projection error of the integrand onto the periodic space inside the box and the cubature error for the projected function in the truncated box.
By making use of this projection, we can measure the non-periodicity of the integrand, and by imposing a decay condition towards zero, our integrand becomes more and more periodic inside the increasing boxes.
Since lattice rules are known to have higher-order convergence for periodic functions, we can achieve higher-order convergence using properly scaled lattice rules for integration over $\R^d$.

We numerically verified our theoretical statements and showed that our method outperforms scaled interlaced Sobol' points for both an example involving the logistic probability density and the normal density, as well as tensor product Gauss--Hermite rules and the sparse grid method based on such Gauss--Hermite rules for the example involving the normal density.
Our scaled lattice rule, which used just a single ``off-the-shelf'' generating vector of a base~$2$ lattice sequence constructed for order~$1$, achieved the convergence rates of order $2$ and $3$ by proper scaling as shown in the theoretical statements.

\section*{Acknowledgement}
The authors acknowledge the support of FWO grant G091920N and NTNU project grant 81617985.

\bibliographystyle{plain}
\bibliography{scaled-lattice-rules-on-Rd}

\end{document}